\documentclass[11pt,letterpaper]{amsart}

\usepackage[margin=1in]{geometry}
\usepackage[
  bookmarks=true]{hyperref}
\usepackage[T1]{fontenc}
\usepackage[english]{babel}
\usepackage[utf8]{inputenc}
\usepackage{csquotes}
\usepackage[final]{microtype}
\usepackage{lmodern}
\usepackage{amsthm}
\usepackage{amssymb}
\usepackage{mathrsfs}
\usepackage{enumerate}
\usepackage{enumitem}
\usepackage{tikz-cd} 
\usepackage{tikz}
\usetikzlibrary{arrows,calc,matrix}
\usepackage[textwidth=1in]{todonotes}
\newtheorem{theorem}{Theorem}[section]
\newtheorem*{theorem*}{Theorem}

\newtheorem{proposition}[theorem]{Proposition}
\newtheorem{lemma}[theorem]{Lemma}
\newtheorem{corollary}[theorem]{Corollary}

\theoremstyle{definition}

\newtheorem{remark}[theorem]{Remark}
\newtheorem{definition}[theorem]{Definition}
\newtheorem*{definition*}{Definition}
\newtheorem{example}[theorem]{Example}

\newcommand{\Spec}{\operatorname{Spec}} 
\newcommand{\Z}{\mathbb{Z}}

\newcommand{\A}{\mathbb{A}} 
\newcommand{\Proj}{\mathbb{P}} 

\newcommand{\GW}{\mathrm{GW}} 
\newcommand{\Adeg}{\deg^{\mathbb{A}^1}} 
\newcommand{\EAdeg}{\deg^{\mathbb{A}^1,\text{Euler}}} 
\newcommand{\Oh}{\mathscr{O}} 
\newcommand{\Tr}{\mathrm{Tr}} 
\newcommand{\ind}{\mathrm{ind}} 
\newcommand{\Hom}{\text{Hom}}

\begin{document}

\title{Global $\A^1$ degrees of covering maps between modular curves}

\author{Hyun Jong Kim and Sun Woo Park}
\address{The Department of Mathematics, University of Western Ontario -- Middlesex College, London, Ontario, Canada, N6A 5B7}
\email{\href{mailto:hyunjongkimmath@gmail.com}{hyunjongkimmath@gmail.com}}
\address{ Department of Mathematics, University of Wisconsin -- Madison, 480 Linconln Dr., Madison, WI 53706, USA}
\email{\href{mailto:swpark2008@gmail.com}{swpark2008@gmail.com}}
\date{\today}

\begin{abstract}
Given a projective smooth curve $X$ over any field $k$, we discuss two notions of global $\mathbb{A}^1$ degree of a finite morphism of smooth curves $f: X \to \mathbb{P}^1_k$ satisfying certain conditions. One originates from computing the Euler number of the pullback of the line bundle $\mathscr{O}_{\mathbb{P}^1_k}(1)$ as a generalization of Kass and Wickelgren's construction of Euler numbers. The other originates from the construction of global $\mathbb{A}^1$ degree of morphisms of projective curves by Kass, Levine, Solomon, and Wickelgren as a generalization of Morel's construction of $\mathbb{A}^1$-Brouwer degree of a morphism $f: \mathbb{P}^1_k \to \mathbb{P}^1_k$. We prove that under certain conditions on $N$, both notions of global $\mathbb{A}^1$ degrees of covering maps between modular curves $X_0(N) \to X(1)$, $X_1(N) \to X(1)$, and $X(N) \to X(1)$ agree to be equal to sums of hyperbolic elements $\langle 1 \rangle + \langle -1 \rangle$ in the Grothendieck-Witt ring $\mathrm{GW}(k)$ for any field $k$ whose characteristic is coprime to $N$ and the pullback of $\mathscr{O}_{\Proj^1}(1)$ is relatively oriented. 
\end{abstract}

\maketitle

\tableofcontents

\section{Introduction}\label{sec: Introduction}
$\A^1$ enumerative geometry uses $\A^1$ homotopy theory over the category of smooth schemes over $k$ to explicitly analyze families of algebraic objects. Morel and Voevodsky constructed the $\A^1$ homotopy category of smooth schemes, which extends homotopy theory to algebraic geometry and views $\A^1$ as a replacement for the unit interval \cite{mv}.  
Applications of $\A^1$ homotopy theory led to generalizations of various profound results in enumerative geometry. Morel (\cite{morel06}), Cazanave (\cite{cazanave}), Kass and Wickelgren (\cite{kwEKL},  \cite{kwBezout}), and many others proposed generalizations of topological degrees of continuous maps as $\A^1$ degrees of morphisms. Kass and Wickelgren used Euler numbers of vector bundles over smooth schemes to count the number of lines on a smooth cubic surface over $\Proj^3$ over any field $k$ with respect to a choice of orientation \cite{kwcubic}. Hoyois proposed a refinement of Grothendieck-Lefschetz trace formula \cite{hoyois}. Levine gave a generalization of Euler characteristic of smooth projective varieties \cite{levine}. Asok and Fasel proved a refinement of Hopf degree theorem and $\A^1$ cohomotopy groups of smooth schemes \cite{af18}. Lastly, Kobin and Taylor extended Kass and Wickelgren's construction of Euler numbers to root stacks \cite{ktstack}.

Arithmetic properties of modular curves, such as the number of components of modular curves, change with respect to different base fields. Snowden gave explicit combinatorial descriptions of real components of modular curves \cite{snowden}. Over the real numbers, the components of $X_0(N)$ comprise of $2^{\omega_{\text{odd}}(N)-1}$ copies of $S^1$ where $\omega_{\text{odd}}(N)$ is the number of distinct odd prime factors of $N$. In Definition \ref{naiveEulerequiv}, we enrich the notion of degrees of finite, relatively orientable maps from a curve to $\Proj^1$. Definition \ref{sum_local_A1} is another enriched notion of degree that emerged in the literature after we studied the former notion. We apply these notions to relatively oriented maps from coarse moduli schemes of elliptic curves with level structures to $X(1)$ using global $\A^1$ degrees of a morphism.

Given a continuous map of orientable manifolds $f: M \to N$ of dimensions $n$ over $\mathbb{R}$, the degree of $f$ is the degree of the induced morphism
\begin{equation*}
    f_* : H_n(M, \mathbb{Z}) \to H_n(N, \mathbb{Z}).
\end{equation*}
The degree of $f$ can be computed from the local degrees of $f$ at each of the fibers of a point $y \in N$. Let $V$ be a small enough open neighborhood of $y$. For each fiber $x \in f^{-1}(y)$, pick an open neighborhood $x \in U \subset f^{-1}(V)$ such that $U \cap f^{-1}(y) = \{x\}$. Then the local degree at $x$ is given by the degree of
\begin{equation*}
    f: S^n \cong U / (U - \{x\}) \to V / (V - \{y\}) \cong S^n.
\end{equation*}
The sum of local degrees at each fiber is equal to the degree of $f$. 

Morel's construction of $\A^1$-Brouwer degree of $f: \Proj^n_k/\Proj^{n-1}_k \to \Proj^n_k/\Proj^{n-1}_k$ extends the notion of degree of continuous functions to the $\A^1$ homotopy category of smooth schemes $X$ over $k$. The $\A^1$-Brouwer degree (or the global $\A^1$ degree) of a morphism $f: \Proj^n_k/\Proj^{n-1}_k \to \Proj^n_k/\Proj^{n-1}_k$ is defined as a morphism
\begin{equation*}
    \Adeg: [\Proj^n_k/\Proj^{n-1}_k, \Proj^n_k/\Proj^{n-1}_k]_{\A^1} \to  \widetilde{CH}^n(\Proj^n_k/\Proj^{n-1}_k) \cong \GW(k),
\end{equation*}
where $\Proj^n_k/\Proj^{n-1}_k$ is the $\A^1$-homotopic $n$-sphere, $\widetilde{CH}^i(X)$ is the $i$th oriented Chow-Witt ring of a smooth scheme $X$ of dimension $n$, and $[ \cdot, \cdot]$ is the set of $\mathbb{A}^1$-homotopy classes of morphisms. Kass and Wickelgren showed that the global $\A^1$ degree of $f$ is the sum of its local $\A^1$ degrees \cite[Proposition 14]{kwEKL}.

In a similar vein, one can extend the above result by considering any finite morphism of form $f: X \to \Proj^n_k \cong \Proj^1_k$ for a smooth projective scheme $X$ of dimension $n$ over $k$. Recall that the Hopf degree theorem states that the degree map 
\begin{equation*}
    \deg: [M, S^n] \to H_n(M, \mathbb{Z}) \cong \mathbb{Z}
\end{equation*}
induces a bijection, where $M$ is a smooth $n$-dimensional manifold over $\mathbb{R}$. Setting $M = S^n$ gives the Hurewicz theorem. In light of this observation, Morel's construction of $\A^1$-Brouwer degree is the analogue of Hurewicz theorem over $\A^1$ homotopy category of smooth schemes over $k$. Furthermore, Asok and Fasel proved the $\A^1$ analogue of Hopf degree theorem formulated by defining an abelian group homomorphism given by
\begin{equation*}
    [X, \Proj^n_k / \Proj^{n-1}_k]_{\A^1} \to \widetilde{CH}^n(X)
\end{equation*}
for any smooth schemes $X$ of dimension $\geq 2$, see \cite[Theorem 2]{adf17} and \cite[Theorem 1]{af18} for further details.

Given a smooth projective curve $C$ over $k$, however, the set of $\A^1$-homotopy classes of morphisms $[C, \Proj^1_k]$ may not necessarily have a functorial abelian group structure. As an alternative, there are two possible notions of defining global $\A^1$ degree of a finite morphism $\pi: C \to \Proj^1_k/\infty$. One, which we introduce in this paper, is constructed from utilizing Euler numbers of a line bundle $\pi^*(\Oh_{\Proj^1_k}(d))$ of $C$. See Definition \ref{def:relative_orientation} for various notions of (relative) orientations that we use in the below definitions. Although the choices of relative orientations are suppressed in the notations below, the below notions of global $\mathbb{A}^1$-degrees depend on such choices.


\begin{definition*}[Global Euler $\A^1$ degree, Definition \ref{naiveEulerequiv}]
Let $\pi: C \to \mathbb{P}^1_k$ be a finite morphism of smooth curves over a field $k$ such that $\pi^* \Oh_{\Proj^1_k} (1)$ is relatively oriented. We can write $\pi = [s_0: s_1]$ where $s_0, s_1$ are global sections of $\pi^* \Oh_{\Proj^1_k}(1)$ with no common zeroes. Let $q \in \Proj^1_k(k)$ be a $k$-rational point, $F_q(X,Y)$ the monic minimal polynomial of $q$, and $Z$ a zero section of $F_q(s_0,s_1)$ considered as a section of $\pi^* \Oh_{\Proj^1_k}(1)$. The global Euler $\A^1$-degree of $\pi: C \to \mathbb{P}^1_k$ is given by
\begin{equation*}
    \EAdeg \pi := e(\pi^* \Oh_{\Proj^1_k}(1), F_q(s_0,s_1)) = \sum_{p \in Z} \ind_{p} F_q(s_0, s_1),
\end{equation*}
where $e(\pi^* \Oh_{\Proj^1_k}(1), F_q(s_0,s_1))$ is the Euler number of the line bundle $\pi^* \Oh_{\Proj^1_k}(1)$ with respect to $F_q(s_0, s_1)$.
\end{definition*}

The other notion originates from work \cite{kass2023quadraticallyenrichedcountrational} by Kass, Levine, Solomon, and Wickelgren that generalizes Morel's construction of $\A^1$-Brouwer degree. We the preceding definition \cite[Theorem 8.7 and Definition 8.8]{PW21} of this notion from Pauli and Wickelgren's lecture notes. 

\begin{definition}[Global $\A^1$-degree] \cite[Theorem 8.7 and Definition 8.8]{PW21}
\label{sum_local_A1}

Let $f: X \to Y$ be a proper map of smooth $d$-dimensional $k$-schemes such that $Tf$ is invertible at some point. Further assume that $f$ is relatively orientable (and equipped with a relative orientation) after removing a codimension $2$ subset of $Y$ and that $Y$ is $\A^1$-chain connected\footnote{A $k$-scheme $Y$ is $\A^1$-chain connected if for any finitely generated separable field extension $L/k$ and any two $L$-points $x,y \in Y(L)$ there are $x = x_0, x_1, \ldots, x_{n-1}, x_n = y \in Y(L)$ and $\gamma_i: \mathbb{A}^1_L \to Y$ with $\gamma_i(0) = x_{i-1}$ and $\gamma_i(1) = x_i$ for $i = 1,\ldots,n$.} with a $k$-point $y$. The global $\A^1$-degree of $f$ is defined as 
$$\deg^{\A^1} f := \sum_{x \in f^{-1}(y)} \deg^{\A^1}_x f$$
and is independent of a generically chosen point $y$.

Here $\deg^{\A^1}_x f$ denotes the local $\A^1$-degree of $f$ at the point $x$, which is defined to be the $\A^1$ Brouwer degree of the composition

\begin{align}
\begin{split}
\mathbb{P}^n_{k}/\mathbb{P}^{n-1}_{k}\simeq T_qX/(T_qX-\{0\})\simeq U/(U-\{q\})\\
\xrightarrow{\bar{f}}Y/(Y- \{p\})\simeq T_pY/(T_pY-\{0\})\simeq \mathbb{P}^n_{k}/\mathbb{P}^{n-1}_{k}.
\end{split}
\end{align}
of morphisms in the $\A^1$ homotopy category of smooth schemes over $k$

\end{definition}

The two notions of global $\A^1$-degrees are generally not equal to each other, cf. Remark \ref{nonex2}. We will later show in Theorems \ref{thm:two_notions_agree} and \ref{mainA} that both global $\A^1$ degrees of a covering $\pi: C \to \mathbb{P}^1_k$ of curves happen to both equal integer multiples of the hyperbolic element $\langle 1 \rangle + \langle -1 \rangle$ in the Grothendieck-Witt ring $\GW(k)$ under certain conditions. We paraphrase these theorems below:

\begin{theorem*}[Theorem \ref{thm:two_notions_agree}] \label{main0}
Let $C$ be a smooth projective curve over a field $k$. Suppose that $\pi: C \to \Proj^1_k$ is a finite morphism satisfying some conditions specified in Theorem \ref{thm:two_notions_agree}. If the pullback of line bundles $\pi^* \Oh_{\mathbb{P}^1}(1)$ and $\pi^* \Oh_{\mathbb{P}^1}(2)$ are relatively oriented, then the global Euler $\A^1$ degree and the global $\A^1$ degrees of $\pi$ are equal to the integer multiple of the hyperbolic element in $\GW(k)$.
\end{theorem*}
\begin{theorem*}[Theorem \ref{mainA}]\label{main1}
Given a fixed integer $N \geq 2$, let $k$ be any field such that $\text{char}(k) \nmid 6N$. Denote by $\pi \in \{\pi_{N}, \pi_{1,N}, \pi_{0,N} \}$ the covering maps of modular curves over $\mathbb{Z}[1/N]$ from $X(N)$, $X_1(N)$, and $X_0(N)$ to $X(1)$, see Definition \ref{cover} for further details. 
Suppose the line bundles $\pi^* \Oh_{\Proj^1_k}(1)$ and $\pi^* T\Proj^1_k$ are relatively oriented. Then the global Euler $\A^1$ degrees and the global $\A^1$ degrees of $\pi$ are equal to integer multiples of the hyperbolic element in $\GW(k)$, except for $\pi = \pi_{0,N}$ where $4 \nmid N$ and none of the prime factors of $q$ of $N$ satisfy $q \equiv 3 \text{ mod } 4$.
\end{theorem*}

The structure of the paper is as follows. In Section 2, we overview relevant results from $\A^1$-homotopy theory, in particular Morel's notion of $\A^1$-Brouwer degree and Kass and Wickelgren's construction of Euler numbers. In Section \ref{sec: Global A1 degrees}, we discuss how one can compute the Euler numbers of the line bundle $\pi^* \Oh_{\Proj^1_k} (n)$ given a choice of a finite morphism of smooth projective curves $\pi: C \to \Proj^1_k$. We define global Euler $\A^1$ degrees of finite morphisms of smooth projective curves $\pi: C \to \Proj^1_k$ under the assumption that $\pi^* \Oh_{\Proj^1_k}(1)$ is relatively oriented over $k$. We then study sufficient conditions under which the global Euler $\A^1$ degrees and global $\A^1$ degrees of finite morphisms $\pi \to \Proj^1_k$ agree. Over such base field $k$, we finally proceed to compute the two notions of global $\A^1$ degrees of covering maps of modular curves in Section \ref{sec: covering maps}. We recall some facts on coarse moduli schemes of elliptic curves with level structures, which we combine with the aforementioned results from Section 2 to Theorem \ref{mainA}. We also list explicit computations of global $\A^1$ degrees of covering maps $\pi_{0,N}:X_0(N) \to X(1)$ for some values of $N$, and discuss what the two variants of global $\A^1$ degrees are for other integers $N$ not divisible by primes congruent to 3 modulo 4.

\subsection{Notation}
Unless otherwise specified, we will notate:
\begin{itemize}
    \item $k$: a field.
    \item $X, Y, Z$: smooth schemes over $k$.
    \item $n$: the dimension of $X$.
    \item $\A^n_k$: the affine $n$-space over $k$.
    \item $\Proj^n_k$: the projective $n$-space over $k$.
    \item $\Oh_X$: the structure sheaf of $X$.
    \item $V \to X$: a vector bundle over $X$.
    \item $\Gamma(X,V)$: the set of global sections of $V \to X$.
    \item $k(X)$: the function field of $X$.
    \item $k(x)$: the residue field of a closed point $x \in X$.
    \item $\Proj^n_k / \Proj^{n-1}_k$: the $\A^1$-homotopic $n$-sphere.
    \item $\EAdeg$: the global Euler-$\A^1$ degree of a finite morphism of smooth projective curves $f: C \to \Proj^1_k$ induced from the Euler number of the pullback of the line bundle $f^* \Oh_{\Proj^1}(1)$. 
    \item $\Adeg$: the $\A^1$-Brouwer degree of $f: \Proj^n_k / \Proj^{n-1}_k \to \Proj^n_k / \Proj^{n-1}_k$. Also denotes the global $\A^1$ degree of a morphism $f: X \to Y$ of smooth proper $k$-schemes when applicable, see Definition \ref{sum_local_A1}.
    \item $[X, Y]$: the set of $\A^1$-homotopy classes of morphisms of smooth schemes $f:X \to Y$ over $k$.
    \item $GW(k)$: the Grothendieck-Witt ring of $k$, see Definition \ref{def:grothendieck_witt_ring}.
    \item $\Tr_{L/k}$: the trace operator $\Tr_{L/k}: \GW(L) \to \GW(k)$ given a separable field extension $L/k$.
    \item $\widetilde{CH}^i(X)$: the $i$-th oriented Chow-Witt ring of a smooth scheme $X$.
    \item $\langle a \rangle$: an element of $GW(k)$ for some $a \in k^\times$.
    \item $e(V,f)$: the Euler number of a vector bundle $V \to X$ given a global section $f \in \Gamma(V,X)$, see Definition \ref{def:eulernumbersection}
    \item $e(\pi, q)$: the Euler number $e(\pi^* \Oh_{\Proj^1_k} (\deg F_q), F_q)$, where $\pi$ is a finite morphism of smooth projective curves $\pi: C \to \Proj^1_k$, $q \in \Proj^1_k$ a closed point, and $F_q$ the monic minimal polynomial of $q$ over $k$, see Definition \ref{def:eulernumber}
    \item $X_{k(p)}$: the base change of a $k$-scheme $X$ to the residue field $k(p)$ of a point $p \in X$.
    \item $p_{k(p)}$: the canonical point of $X_{k(p)}$ determined by $p: \operatorname{Spec} k(p) \to X$ where $X$ is a $k$-scheme.
    \item $N$: the level of modular curves.
    \item $\omega_{\text{odd}}(N)$: the number of distinct odd prime factors of $N$.
    \item $X(N), X_0(N), X_1(N)$: modular curves of level $N$ considered as smooth projective schemes over $\Z[1/N]$.
    \item $Y(N), Y_0(N), Y_1(N)$: affine modular curves of level $N$ considered as smooth affine schemes over $\Z[1/N]$.
    \item $E$: an elliptic curve over $k$.
    \item $j$: the classical $j$-invariant of an elliptic curve $E$.
\end{itemize}

We define these notations more precisely in the upcoming sections.

\section{\texorpdfstring{$\A^1$}{A1} homotopy theory} \label{sec: A1 homotopy theory}

We introduce $\A^1$ homotopy theory by listing some essential definitions and results. The concepts that we list in this section follow \cite{AWS}, \cite{kwcubic}, and \cite{kwBezout}, and lead up to Euler numbers, see Definition \ref{def:eulernumbersection}. Euler numbers are sums of local indices. We think of Euler numbers as analogous to global $\A^1$-degrees and local indices as analogous to local $\A^1$-degrees, cf. Example \ref{ex:euler_number_A1_degree}.

\subsection{\texorpdfstring{$\A^1$}{A1} degrees of finite morphisms over projective spaces}\label{subsec: A1 degree basic}
We first introduce the construction of $\A^1$ degrees of finite self-morphisms of $\mathbb{A}^n$ or $\mathbb{P}^n$. This construction generalizes that of the topological Brouwer degree of continuous maps between manifolds.

\begin{definition}
Let $X$ be a scheme over $k$. A Nisnevich neighborhood of $x \in X$ is an \'etale morphism $p: U \rightarrow X$ such that there is some $u \in U$ such that $p(u) = x$ and the induced map $k(x) \rightarrow k(u)$ of residue fields is an isomorphism.
\end{definition}

Nisnevich sites on categories of smooth schemes over $k$ are finer than Zariski sites but coarser than \'etale sites.

\begin{definition}
Let $X$ be a scheme of dimension $n$. An \'etale map $\phi: U \rightarrow \A^n$, where $U$ is a Zariski open neighborhood of $p \in X$, is referred to as Nisnevich local coordinates around $p$ if the induced map of residue fields $k(\phi(p)) \rightarrow k(p)$ is an isomorphism.
\end{definition}

\begin{proposition}\label{curvenis}
Let $X$ be a smooth scheme over $k$ of dimension $n \geq 1$. There are Nisnevich coordinates near any closed point.
\end{proposition}
\begin{proof}
We refer to \cite[Proposition 20]{kwcubic}.
\end{proof}

\begin{definition} \label{def:grothendieck_witt_ring}
Let $k$ be a field. The isometry classes of non-degenerate, symmetric, bilinear forms over $k$ form a semi-ring under $\oplus$, the direct sum operator, and $\otimes$, the tensor product operator. The Grothendieck-Witt group $\GW(k)$ is the group completion of this semi-ring under $\oplus$.
The Grothendieck-Witt group is generated by elements of the form $\langle a \rangle$, where $\langle a \rangle$ represents the isometry class of the bilinear form $k \times k \rightarrow k, (x,y) \mapsto axy$. The generators are subject to the following relations:
\begin{enumerate}
    \item $\langle a \rangle = \langle b^2 a \rangle$ for $b \in k^\times$
    \item $\langle a \rangle \langle b \rangle = \langle ab \rangle$
    \item $\langle u \rangle + \langle v \rangle = \langle uv(u+v) \rangle + \langle u+v \rangle$
    \item $\langle u \rangle + \langle -u \rangle = \langle 1 \rangle + \langle -1 \rangle$.
\end{enumerate}
\end{definition}

In fact, global $\A^1$ degrees of morphisms $f: \Proj^n_k \to \Proj^n_k$ have been carefully studied, for instance by Morel \cite{morel06}, Kass and Wickelgren \cite{kwEKL}, and Cazanave \cite{cazanave}. We first state the definition of the trace operator on Grothendieck-Witt rings, whose construction we will use to state a proposition which claims that the global $\A^1$ degrees of such morphisms is the sum of local $\A^1$ degrees at the fibers.

\begin{definition} \label{def:trace}
Let $L/k$ be a finite separable field extension with a trace functor $\Tr_{L/k}: L \rightarrow k$. Define the trace operator $\Tr_{L/k}: \GW(L) \rightarrow \GW(k)$ by $[\beta] \mapsto [\Tr_{L/k} \circ \beta]$ where $\beta: V \times V \rightarrow L$ is a non-degenerate symmetric bilinear form.
\end{definition}

In particular, Kass and Wickelgren prove that there exists an algebraic method to explicitly compute the local $\A^1$ degree at a fiber of $f: \A^n_k \to \A^n_k$. Here, we define the Eisenbud-Khimshiashvili-Levine class of a morphism $\A^n_k \rightarrow \A^n_k$.


\begin{definition} \label{def:ekl}
Let $R$ be the polynomial ring $k[x_1,\ldots,x_n]$. Let $f = (f_1,\ldots,f_n): \A^n_k \rightarrow \A^n_k$ be a finite $k$-morphism and let $x \in \A^n_k$ be a closed point such that $y = f(x)$ has residue field $k$. Express the coordinates of $y$ by $y = (\overline{b}_1,\ldots,\overline{b}_n)$. Furthermore, we denote by $\mathfrak{m}_x$ the maximal ideal of $R$ corresponding to $x$.

Define the local algebra $Q_x(f)$ of $f$ at $x$ to be the local ring $R_{\mathfrak{m}_x}/(f_1-\overline{b}_1,\ldots,f_n-\overline{b}_n)$. The local algebra at the origin is $Q_0(f)$. 

Let $a_{i,j} \in R$ be polynomials satisfying
$$
    f_i(x) = f_i(0) + \sum_{j=1}^n a_{i,j} x_j.
$$
The distinguished socle element is the element $E_0(f) = \det(a_{i,j})$ of $Q_0(f)$. 

Given any $k$-linear map $\phi: Q_0(f) \rightarrow k$ satisfying $\phi(E_0(f)) = 1$, define $\beta_\phi: Q_0(f) \times Q_0(f) \rightarrow k$ to be the symmetric bilinear form $\beta_\phi(a_1,a_2) = \phi(a_1 \cdot a_2)$. The Grothendieck-Witt class of Eisenbud-Khimshiashvili-Levine or the EKL class $w_0(f) \in \GW(k)$ is the Grothendieck-Witt class of $\beta_\phi$. By \cite[Lemma 6]{kwEKL}, this is independent of the choice of $\phi$.
\end{definition}

A crucial theorem of Kass and Wickelgren in \cite{kwEKL} states the following.
\begin{theorem}\label{localEKL}\cite[Main Theorem]{kwEKL}
Let $f: \A^n_k \to \A^n_k$ be a finite morphism of $k$-schemes. Then the local $\A^1$ degree of $f$ at a $k$-rational point $x \in \A^n_k(k)$ is equivalent to the EKL class $w_0(f)$.
\end{theorem}

Other than the EKL class, we can also use the B\'ezout degree to compute the global $\A^1$ degree of the morphism $f: \Proj^n_k \to \Proj^n_k$. In particular, \cite[Lemma 6]{kwBezout} equates the Grothendieck-Witt class of a certain type of nondegenerate symmetric bilinear form with the $\A^1$ degree of maps of Proposition \ref{adegpoly}.

\begin{lemma} \label{bezoutlemma}
Let $k$ be any field, and let $A_1, A_2, \cdots, A_n \in k$ be scalars such that $A_n$ is nonzero. Then the nondegenerate symmetric matrix 
\begin{equation*}
    M = \begin{pmatrix}
    A_1 & A_2 & \cdots & A_{n-1} & A_n \\
    A_2 & A_3 & \cdots & A_n & 0 \\
    \vdots & \vdots & \ddots & \vdots & \vdots \\
    A_{n-1} & A_n & \cdots & 0 & 0 \\
    A_n & 0 & \cdots & 0 & 0
    \end{pmatrix}
\end{equation*}
corresponds to the following element in $\GW(k)$:
\begin{equation*}
    M = \begin{cases}
    \langle A_n \rangle + \frac{n-1}{2} \left( \langle 1 \rangle + \langle -1 \rangle  \right) \; &\text{if } \; n \; \text{ odd} \\
    \frac{n}{2} \left( \langle 1 \rangle + \langle -1 \rangle \right) \; &\text{if } \; n \; \text{ even}.
    \end{cases}
\end{equation*}
\end{lemma}

\begin{definition}
The B\'ezout degree of a finite morphism $\frac{F}{G}: \Proj^1_k \to \Proj^1_k$ is represented by the matrix $\left( b_{i,j} \right)_{i,j}$ where $b_{i,j}$ are defined as follows:
\begin{equation*}
    F(x)G(y) - F(y)G(x) = (x-y) \left( \sum_{i,j} b_{i,j} x^{i-1} y^{j-1} \right).
\end{equation*}
We call the matrix $\left(b_{i,j} \right)_{i,j}$ the B\'ezout matrix of $\frac{F}{G}$.
\end{definition}

\begin{theorem} \label{thm:globalA1Bezout}
The global $\A^1$ degree of a finite morphism $\frac{F}{G}: \Proj^1_k \to \Proj^1_k$ is equivalent to the B\'ezout degree of $\frac{F}{G}$.
\end{theorem}
\begin{proof}
We refer to \cite[Theorem 3]{kwBezout} and \cite[Theorem 1.2]{cazanave}.
\end{proof}
We will use the theorem when we compute explicit examples of global $\A^1$ degrees of covering maps of modular curves equivalent to $\mathbb{P}^1_k$.

We also state the following result from \cite[Lemma 5]{kwBezout}, which computes the $\A^1$ degree of maps $\frac{x^n}{c}: \mathbb{P}^1_k \rightarrow \mathbb{P}^1_k$.
\begin{proposition}\label{adegpoly} \cite[Lemma 5]{kwBezout}
For $c \in k^*$, 
\begin{equation*}
    \deg^{\A^1} \left( \frac{x^n}{c} \right) = \begin{cases} \langle c \rangle + \frac{n-1}{2} \cdot (\langle 1 \rangle + \langle -1 \rangle) &\text{if } n \text{ is odd} \\ \frac{n}{2} \cdot (\langle 1 \rangle + \langle - 1 \rangle) & \text{if } n \text{ is even}  \end{cases}
\end{equation*}
\end{proposition}

We end the subsection with the following proposition, which simplifies the computation of local $\A^1$ degree of a morphism $f$ at a fiber with double ramification.
\begin{lemma}\label{TraceDoubleRamify}
Let $k$ be any field of characteristic coprime to $2$. Let $L$ be any finite separable field extension over $k$. Then the trace of the hyperbolic element $\langle 1 \rangle + \langle -1 \rangle$ is given by
\begin{equation*}
    \Tr_{L/k} (\langle 1 \rangle + \langle -1 \rangle) = [L:k] \left( \langle 1 \rangle + \langle -1 \rangle \right).
\end{equation*}
\end{lemma}
\begin{proof}
Let $\alpha$ be a primitive element of $L$. Let $f(x)$ be the monic minimal polynomial of $\alpha$. Pick $\eta$, a bilinear form over $L[x]/(x-\alpha)^2$ such that $\eta(x) = 1$ and $\eta(1) = 0$. Then $\eta$ induces a non-degenerate quadratic form $\beta$ over $L$, considered as a $k$-vector space. The equation for $\beta$ is given by
\begin{equation*}
    \beta(a, b) := \Tr_{L/k}(\eta(ab)).
\end{equation*}
We assume that the basis of $L$, as a $k$-vector space, is the set of monomials $\{1, x, \cdots, x^{[L:k]}\}$. It is clear to check that for any polynomial $g(x) \in L$,
\begin{equation*}
    \beta(g(x), g(x)) = \Tr_{L/k} ( \eta(g(x)^2))
\end{equation*}
This implies that any polynomial of form $g(x) = f(x)h(x)$ where $(f, h) = 1$ is an isotropic vector of $\beta$. Using this observation, we can explicitly construct an orthogonal basis $B$ of a maximal totally isotropic subspace of dimension $\frac{[L:k]}{2}$:
\begin{equation*}
    B := \{f(x), x f(x), x^2 f(x), \cdots, x^{n-1} f(x)\}.
\end{equation*}
We can hence conclude that $\beta$ is hyperbolic. 
\end{proof}

\subsection{Euler numbers of vector bundles over smooth schemes}\label{subsec: Euler numbers}

 We present an exposition on Euler numbers of vector bundles over smooth schemes, constructed by Kass and Wickelgren \cite{kwcubic}. The terminologies presented in this subsection closely follow those in \cite[Section 4]{kwcubic} and \cite[Section 4]{AWS}.
 
For this subsection, let $X/k$ be a smooth dimension $n$ scheme over a field $k$ unless otherwise specified. Furthermore, let any vector bundles $V$ on $X$ be of rank $n$.

Following \cite{okotel14}, \cite{kwcubic} defines relative orientations of lines bundles. Moreover, \cite{kass2023quadraticallyenrichedcountrational} defines relative orientations of lci morphisms between schemes.

\begin{definition}{(cf. \cite[Definition 17]{kwcubic}, \cite[Definition 2.2]{kass2023quadraticallyenrichedcountrational})} \label{def:relative_orientation}
Let $X/k$ be a smooth $n$-dimensional scheme over a field $k$ and let $V$ be a vector bundle on $X$ of rank $n$. We say that $V$ is orientable if there is some line bundle $L$ on $X$ and some isomorphism $\det V := \wedge^n V \cong L^{\otimes 2}$. An orientation of $V$ consists of such an $L$ and such an isomorphism. We say that $V$ is oriented if an orientation of $V$ is fixed.

We say that $V$ is relatively orientable there is some line bundle $L$ on $X$ and some isomorphism $\Hom(\wedge^{n} TX, \wedge^{n} V) \cong L^{\otimes 2}$ where $TX$ is the tangent bundle of $X$. A relative orientation of $V$ consists of such an $L$ and such an isomorphism, and we say that $V$ is relatively oriented if a relative orientation of $V$ is fixed.

Given a local complete intersection morphism $f: X \to Y$, let $L_f$ be its cotangent complex, which is perfect \cite[Tag 08SL]{stacks-project}. Let $\omega_f$ denote $\det L_F$, the determinant of $L_f$, which is a line bundle on $X$. We say that $f$ is relatively orientable if there exists a line bundle $L$ on $X$ and an isomorphism $\omega_f \cong L^{\otimes 2}$. A relative orientation consists of such an $L$ and such an isomorphism, and we say that $f$ is relatively oriented if a relative orientation is fixed. 
\end{definition}

\begin{remark}
    Any locally of finite type morphism between regular schemes is a local complete intersection morphism \cite[Tag 0E9K]{stacks-project}. The lecture notes \cite{PW21} discussed global $\mathbb{A}^1$-degrees for relatively oriented morphisms $f: X \to Y$ of smooth (proper) $k$-schemes before \cite{kass2023quadraticallyenrichedcountrational} was available in the literature --- in this case, $f$ is relatively orientable exactly when $\operatorname{Hom}(\det TX, f^* \det TY)$ is isomorphic to the square of a line bundle, cf. \cite[Definition 8.1]{PW21}.
\end{remark}

Pick Nisnevich local coordinates $\phi: U \to \A^n_k$ at a closed point $p \in X$. Given a relatively oriented vector bundle $V \to X$, we can choose a sufficiently small open $U$ such that $V$ is trivial when restricted to $U$. 
\begin{definition} \cite[Definition 21]{kwcubic}
Let $V \to X$ be a vector bundle with a relative orientation $L^{\otimes 2} \cong \text{Hom} \left( \bigwedge\nolimits^n \text{T}X, \bigwedge\nolimits^n V \right)$. A trivialization of $V|_U$ is said to be compatible with the Nisnevich coordinates $\phi$ and the relative orientation if there exists an element $f \in \Gamma(U,L)$ such that $f^{\otimes 2} \in \Gamma(U, L^{\otimes 2})$ sends a distinguished basis of $\bigwedge\nolimits^n \text{T}X|_U$ to a distinguished basis of $\bigwedge\nolimits^n V|_U$. 
\end{definition}

Given a choice of a Nisnevich local coordinate $\phi: U \to \A^n_k$ and a compatible trivialization of $V|_U$, Kass and Wickelgren define the Euler number of a vector bundle given a choice of a global section. 

\begin{definition} \cite[Definition 22]{kwcubic}
Let $V \to X$ be a vector bundle of rank $n = \dim X$, and let $f \in \Gamma(X, V)$ be a global section. Let $Z \subset X$ be the closed subscheme defined as the zero set of $f$. Then a closed point $p \in X$ is an isolated zero of $f$ if $\Oh_{Z,p}$ is a finite $k$-algebra. We say that $f$ has isolated zeros if $\Oh_Z$ is a finite $k$-algebra.
\end{definition}

In \cite[Proposition 23]{kwcubic}, Kass and Wickelgren gives a criterion of whether a global section $f$ has isolated zeros, which we state as below.

\begin{proposition} \label{isozero} \cite[Proposition 23]{kwcubic}
A global section $f \in \Gamma(X,V)$ of a vector bundle $V \to X$ of rank $\dim X$ has isolated zeros if and only if $Z$, the closed subscheme of $X$ defined as the zero set of $f$, consists of finitely many points.
\end{proposition}

Given a global section $f \in \Gamma(X,V)$ with isolated zeros, we define the local index of $f$ at each isolated zero $p$ defined over the field $k$.

\begin{definition} \cite[Definition 30]{kwcubic} \label{def:localind}
    Let $V \to X$ be a relatively oriented vector bundle of rank $\dim X$. Let $f \in \Gamma(X,V)$ be a section. Let $p \in Z$ be any isolated zero of $f$. Choose Nisnevich coordinates $\phi: U \to \mathbb{A}^r_k$ around $p$. Choose a trivialization of $V|_U$ that is compatible with $\phi$ and the relative orientation.
    
    There is a presentation $\mathcal{O}_{Z,p} \cong k[x_1,\ldots,x_r]/\langle g_1,\ldots,g_r \rangle$ yielding a canonical isomorphism $\operatorname{Hom}_k(\mathscr{O}_{Z,p}, k) \cong \mathscr{O}_{Z,p}$ cf. \cite[Section 4, the discussion after Lemma 24 up to Lemma 28]{kwcubic}. Let $\eta$ be the element of $\operatorname{Hom}_k(\mathscr{O}_{Z,p}, k)$ which corresponds to $1$ in $\mathscr{O}_{Z,p}$. Let $\beta$ be the symmetric bilinear form defined by 
    
    $$\beta(x,y) = \eta(xy)$$
    
    The local index $\ind_p f$ of $f$ at $p$ is the equivalence class in the Grothendieck-Witt ring represented by the symmetric bilinear form $\beta$.
    
    The index $\ind_p f$ does not depend on the choice of $\phi$, the compatible trivialization, and the $g_i$ \cite[Corollary 31]{kwcubic}. Moreover, if $p$ is a $k$-point, then one may replace $\eta$ by any $k$-linear homomorphism $\eta_{\text{new}}: \mathscr{O}_{Z,p} \to k$ which takes the distinguished socle element to $1$h \cite[Section 4, the discussion after Corollary 17]{kwcubic}.
    
\end{definition}

As proved in \cite[Section 4]{kwcubic}, the local ring $\Oh_{Z,p}$ is of finite complete intersection, hence its socle is a dimension $1$ vector space over $k$. This allows us to find the distinguished socle element in $\Oh_{Z,p}$. We also note that $\ind_p f$ which represents $\beta$ is independent of the choice of $\eta$, the choice of Nisnevich coordinates of $X$, and the choice of a compatible trivialization of the vector bundle $E$, see \cite[Corollary 31]{kwcubic} for further details.

The above definition, which holds for any $k$-rational points of isolated zeros of $f$, can be extended to any finite separable field extension $L$ over $k$, as shown in \cite[Proposition 34]{kwcubic}.

\begin{proposition} \cite[Proposition 34]{kwcubic} \label{local_index_base_change}
Let $V$ be a relatively oriented vector bundle of rank $\dim X$ over $X$. Let $f \in \Gamma(X,V)$ be a section of $V$. Let $p$ be an isolated zero of $f$ whose residue field $k(p)$ is separable over $k$.

Then,
\begin{equation*}
    \ind_p f = \Tr_{k(p)/k} (\ind_{p_{k(p)}} f_{k(p)})
\end{equation*}
where $f_{k(p)}$ denotes the base change of $f$ to $k(p)$ and $p_{k(p)} \hookrightarrow X_{k(p)}$ denotes the canonical point of $X_{k(p)}$ determined by $p: \operatorname{Spec} k(p) \to X$.
\end{proposition}

Using local indices, we define the Euler number of a global section $f \in \Gamma(X,V)$.

\begin{definition} \cite[Definition 35]{kwcubic} \label{def:eulernumbersection}
Let $V \to X$ be a relatively orientable vector bundle with a fixed relative orientation $L^{\otimes 2} \cong \operatorname{Hom}(\wedge^n TX, \wedge^n V)$, and let $f \in \Gamma(X, V)$ be a section of $V$ with isolated zeroes whose residue fields are separable over $k$. Then the Euler number of $V$ given a global section $f$ is defined as 
\begin{equation*}
    e(V, f) := \sum_{p \in Z} \ind_p f
\end{equation*}
where $Z$ is the zero locus of $f$.
\end{definition}

\begin{remark}
    The definition of $e(V,f)$ depends on the relative orientation even though we suppress this relative orientation from the notation.
\end{remark}
 
Under certain conditions, the Euler number is independent of the choice of a global section of $V$. Section \ref{sec: Global A1 degrees} further discusses the invariance of choice of global sections. 

\begin{example} \label{ex:euler_number_A1_degree}
When $X = \A^n_k$ and $V = \Oh^n_X$, the Euler number of $V$ given a global section $f \in \Gamma(X, V)$ is equivalent to the global $\A^1$ degree of $f: \A^n_k \to \A^n_k$. The local indices of the Euler number is equal to the local $\A^1$ degrees of $f$, which are determined by the EKL class of $f$ at each fiber of $0 \in \A^n_k$ by Theorem \ref{localEKL}. Further details can be found in \cite[Example 32]{kwcubic} and \cite[Main Theorem]{kwEKL}.
\end{example}

\section{Two variants of global \texorpdfstring{$\A^1$}{A1} degrees of finite maps into the projective line}
\label{sec: Global A1 degrees}

In this section, we define some notions of Euler numbers and global Euler $\A^1$-degrees for some finite maps of schemes. Section \ref{sec: Relative orientability} does so for morphisms $C \rightarrow \Proj^1_k$. Section \ref{subsec: euler number into affine} does so for morphisms $Y \rightarrow \A^n_k$. Finally, Section \ref{subsec: explicit method} roughly reduces computations of local indices for morphisms $C \rightarrow \Proj^1_k$ into computations of local indices for morphisms $\A^r_k \rightarrow \A^r_k$. This allows us to pinpoint the conditions to guarantee the equality between global Euler $\A^1$-degrees and global $\A^1$-degrees of morphisms $C \rightarrow \Proj^1_k$. Example \ref{example:ellipticcurve} shows that this rough reduction can sometimes be sufficient to precisely determine both notions of global $\A^1$ degrees.

\subsection{Defining Euler numbers of finite maps into the projective line}\label{sec: Relative orientability}

Throughout this subsection, unless otherwise specified, let $\pi: C \rightarrow \Proj^1_k$ be a finite morphism over a field $k$ whose characteristic is not $2$ where $C$ is a smooth projective curve and let $q \in \Proj^1_k$ be a closed point. Express $\pi$ as $\pi = [s_0:s_1]$ where $s_0,s_1$ are global sections of $\pi^* \mathscr{O}_{\Proj^1_k}(1)$ with no common zeroes. Given a closed point $q \in \Proj^1_k$, let $F_q(X,Y)$ be the ``monic minimal'' polynomial of $q$ such that either $F_q(X,Y)$ monic in $X$ if $q \neq \infty$, or $F_q(X,Y) = Y$ if $q = \infty$. From here, we simply say that $F_q(X,Y)$ is the monic minimal polynomial of $q$. In particular, $F_q(X,Y)$ is a nonzero homogeneous irreducible polynomial over $k$ and $F_q(s_0,s_1)$ can be interpreted a section of $\pi^* \mathscr{O}_{\Proj^1_k}(\deg F_q)$.

In this subsection, we define $e(\pi, q)$ when $\pi^* \Oh_{\Proj^1}(\deg q)$ is relatively oriented. Corollary \ref{euler deg independent} shows that $e(\pi, q)$ only depends on $\deg q$, and define the global Euler $\A^1$ degree as the Euler number $e(\pi, q)$. Moreover, when $\pi$ is has some fiber consisting entirely of doubly ramified points, we show that the Euler number is an appropriate multiple of the hyperbollic element $\mathbb{H} = \langle 1 \rangle + \langle -1 \rangle$. \par

More specifically, in Definition \ref{def:eulernumber} below, we define $e(\pi,q)$ by identifying it with the Euler number of the section $F_q(s_0,s_1)$ of the line bundle $\pi^*(\mathscr{O}_{\Proj^1_k}(\deg F_q))$. Lemma \ref{lemma:fiber} below identifies the fiber $\pi^{-1}(q)$ with the zero locus of $F_q(s_0, s_1)$. The Euler number of a section $\sigma$ of a line bundle is defined as the sum of local indices on the zero locus points of $\sigma$. Therefore, $e(\pi,q)$ can be computed as the sum of local indices at the points of $\pi^{-1}(q)$ --- this aligns with the intuition that global degrees can be computed as the sum of local degrees. 

\begin{lemma} \label{lemma:fiber}
Let $q \in \Proj^1_k$ be a closed point. The zero locus of the section $F_q(s_0,s_1)$ is the fiber of $\pi$ at $q$.
\begin{proof}
The zero locus of $F_q(s_0,s_1)$ can be locally determined on open subschemes of $C$ trivializing $\pi^* \mathscr{O}_{\Proj^1_k}(\deg F_q)$. Let $X,Y$ denote the coordinates of $\Proj^1_k$ so that $\A^1_{X/Y} := \Spec(k[X/Y])$ and $\A^1_{Y/X} := \Spec( k[Y/X])$ are open subschemes of $\Proj^1_k$ that trivialize $\mathscr{O}_{\Proj^1_k}(\deg F_q)$ and hence $\pi^{-1}(\A^1_{X/Y})$ and $\pi^{-1}(\A^1_{Y/X})$ trivialize $\pi^* \mathscr{O}_{\Proj^1_k}(\deg F_q)$. Say that $q \in \Spec(k[X/Y]) = \A^1_{X/Y} \subset \Proj^1_k$. The fiber $\pi^{-1}(q)$ is the Cartesian product $C \otimes_{\Proj^1_k} q$, which restricts to $\pi^{-1}(\A^1_{X/Y}) \otimes_{\A^1} q$ over $\pi^{-1}(\A^1_{X/Y})$. Note that $q = \Spec(k[X/Y]/(F_q(X/Y,1)))$ and that $\pi|_{\pi^{-1}(\A^1_{X/Y})}$ is the map $\pi^{-1} (\A^1_{X/Y}) \xrightarrow{s_0/s_1} \A^1_{X/Y}$.
\begin{center}
\begin{tikzcd}
 \pi^{-1}(\A^1_{X/Y}) \otimes_{\A^1_{X/Y}} q \ar[r] \ar[d] & q = \Spec(k[X/Y]/(F_q(X/Y,1))) \ar[d] \\
 \pi^{-1}(\A^1_{X/Y}) \ar[r,"\pi|_{\pi^{-1}(\A^1_{X/Y})} = s_0/s_1"] & \A^1_{X/Y}
\end{tikzcd}
\end{center}
Therefore, $\pi^{-1}(\A^1_{X/Y}) \otimes_{\A^1} q$ is the closed subscheme of $\pi^{-1} (\A^1_{X/Y})$ cut out by $F_q(s_0/s_1,1)$. On the other hand, the section $F_q(s_0,s_1)$ restricts to $F_q(s_0/s_1,1)$ on $\pi^{-1}(\A^1_{X/Y})$. Thus, the zero locus of $F_q(s_0,s_1)$ and the fiber $\pi^{-1}(q)$ agree on $\pi^{-1}(\A^1_{X/Y})$. By symmetry, they agree on $\pi^{-1}(\A^1_{Y/X})$ as well, so they are the same closed subscheme of $C$.
\end{proof}
\end{lemma}

\begin{definition} \label{def:eulernumber}
Let $q \in \Proj^1_k$ be a closed point. Assume that $\pi^*(\mathscr{O}_{\Proj^1_k}(\deg F_q))$ is a relatively oriented line bundle of $C$. Define the Euler number $e(\pi,q)$ to be the Euler number $e(\pi^* \mathscr{O}_{\Proj^1_k}(\deg F_q), F_q(s_0,s_1))$. 
\end{definition}

A priori, $e(\pi,q)$ might depend on the choice of $F_q(s_0,s_1)$ up to scalar multiplication, but Proposition \ref{proposition:eulernumberpoly} below shows that this is not the case.

Furthermore, the relative orientation of the line bundle $\pi^* \Oh_{\Proj^1_k}(\deg F_q)$ ensures that $e(\pi^* \mathscr{O}_{\Proj^1_k}(\deg F_q))$ is well defined, cf.~Definition \ref{def:eulernumbersection}.

For $\pi^*\mathscr{O}_{\Proj^1_k}(\deg F_q)$ to be relatively orientable, it must be of even degree, i.e. $\deg \pi^* \mathscr{O}_{\Proj^1_k} (\deg F_q) = \deg \pi \cdot \deg F_q$ is even. Conversely, if $L = \pi^* \mathscr{O}_{\Proj^1_k}(\deg F_q)$ is of even degree, then $\Hom(\wedge^{\text{top}} TC, \wedge^{\text{top}} L)$ is also of even degree and hence orientable over $\overline{k}$. There is thus some finite extension $k'$ of $k$ over which the base change of $\Hom(\wedge^{\text{top}} TC, \wedge^{\text{top}} L)$ is orientable. Letting $\pi_{k'}: C_{k'} \rightarrow \mathbb{P}^1_{k'}$ denote the base change of $\pi$ via $\Spec(k') \rightarrow \Spec(k)$, the line bundle $\pi^*_{k'} \mathscr{O}_{\Proj^1_{k'}}(\deg F_q)$ of $C_{k'}$ is relatively orientable. In Section \ref{sec: covering maps} we will calculate the Euler numbers of even-degree covering maps $\pi: C \rightarrow X(1) \cong \mathbb{P}^1_\mathbb{Q}$, where $C$ is a coarse moduli scheme, after base change to a number field $k'$ over which $\pi^* \mathscr{O}_{\Proj^1_{k'}}(1)$ is relatively orientable. Note that the choice of $k'$ may depend on $\pi$. 

The discussion of the above paragraph still holds when $\overline{k}$ is replaced with $k^\text{sep}$ when $\text{char } k \neq 2$ because the multiplication-by-$2$ map on $\text{Pic}(C)$ is a separable isogeny and hence the $\overline{k}$-points of $\text{Pic}(C)$ form a $2$-divisible group, cf. \cite[Proposition 5.9 and Corollary 5.10]{vm11}

\begin{remark}
Although $ \pi^* \mathscr{O}_{\Proj^1_k} (\deg F_q)$ is relatively orientable after base change to both $k'$ and $\overline{k}$, we only use $k'$ to discuss Euler numbers because the Grothendieck-Witt ring of $\overline{k}$ is ``trivial'', by which we mean that $\GW(\overline{k}) \cong \mathbb{Z}$. In this case, $\GW(\overline{k})$ is the (additive) group completion of the equivalence classes of bilinear forms of finite dimensional vector spaces over $\overline{k}$, but such equivalence classes are entirely determined by the ranks of their bilinear forms. Accordingly, the Euler number $E(\pi_{\overline{k}},q)$ must equal $\deg \pi \cdot \deg q \in \mathbb{Z} \cong \GW(\overline{k})$
\end{remark}



We now show that the Euler number $e(\pi,q)$ only depends on $\deg q$. To do so, we use the notion of sections of a line bundle connected by sections with isolated zeros, as introduced in \cite[Definition 37]{kwcubic}. Under certain circumstances, such sections admit the same Euler number by Theorem \ref{thm:eulernumberconnectedsection} below.

\begin{definition} \label{def:sectionsconnected}
Let $\pi: E \rightarrow X$ be a vector bundle of rank $r = \dim X$ where $X$ is a smooth, proper scheme over $k$. Denote $\mathcal{E}$ to be the pullback of $E$ to $X \times \mathbb{A}^1$. We say that sections $\sigma,\sigma'$ with isolated zeros can be connected by sections with isolated zeros if there exist sections $s_i$ for $i = 0,1,\ldots,N$ of $\mathcal{E}$ and rational points $t_i^-$ and $t_i^+$ of $\A^1$ for $i = 1,\ldots, N$ such that
\begin{enumerate}[label=(\arabic*)]
	\item for $i = 0,\ldots, N$, and all closed points $t$ of $\A^1$, the section $(s_i)_t$ of $E$ has isolated zeros,
	\item $(s_0)_{t_0^-}$ is isomorphic to $\sigma$, 
	\item $(s_N)_{t_N^+}$ is isomorphic to $\sigma'$, and
	\item $(s_i)_{t_i^+}$ is isomorphic to $(s_{i+1})_{t_{i+1}^-}$ for $i = 0,\ldots,N-1$. 
\end{enumerate}
Here, given a section $s$ of $\mathcal{E}$ and a rational point $t \in \A^1$, $s_t$ denotes the pullback of $s$ under the embedding $X  \times_k \Spec(k(t)) \xrightarrow{\text{id},t} X \times \mathbb{A}^1$.
\end{definition}

\begin{theorem}{\cite[Corollary 38]{kwcubic}} \label{thm:eulernumberconnectedsection}
Let $\pi: E \rightarrow X$ be a relatively oriented vector bundle on a smooth, proper dimension $r$ scheme $X$ over $k$. Suppose that $\sigma$ and $\sigma'$ are sections of $\pi$ with isolated zeros and let $\sigma_L$ and $\sigma'_L$ be their base changes by an odd degree field extension $L$ of $k$. If $\sigma_L$ and $\sigma'_L$ can be connected by sections with isolated zeros, then $e(E,\sigma) = e(E,\sigma')$. \par
In particular, if any two sections of $\pi$ with isolated zeros can be connected by sections with isolated zeros after base change by an odd degree field extension of $k$, then $e(E,\sigma) = e(E,\sigma')$ for all sections $\sigma$ and $\sigma'$ with isolated zeros.  
\end{theorem}

\begin{proposition} \label{proposition:eulernumberpoly}
Let $F(X,Y)$ and $F'(X,Y)$ be nonzero monic homogeneous polynomials over $k$ of degree $d$. Assume that $\pi^*\mathscr{O}_{\mathbb{P}^1_k}(d)$ is relatively oriented, so that Euler numbers of sections can be defined. \par
The sections $F(s_0,s_1)$ and $F'(s_0,s_1)$ of $\pi^*\mathscr{O}_{\mathbb{P}^1_k}(d)$ have equal Euler numbers, i.e.~
$$e(\pi^* \mathscr{O}_{\mathbb{P}^1_k}(d), F(s_0,s_1)) = e(\pi^* \mathscr{O}_{\mathbb{P}^1_k}(d), F'(s_0,s_1)).$$
\begin{proof}

Note that $F(s_0,s_1)$ and $F'(s_0,s_1)$ are connected via the section $(1-t) F(s_0, s_1) + t F'(s_0, s_1)$ of the pullback of $\pi^* \mathscr{O}_{\mathbb{P}^1_k}(d)$ under the projection map $C \times \mathbb{A}^1 \to C$. Moreover, for each closed point $t_0 \in \mathbb{A}^1_L$ for any base change $L$ of odd degree over $k$, the pullback $(1-t_0) F(s_0, s_1) + t_0 F'(s_0, s_1)$ of $(1-t) F(s_0, s_1) + t F'(s_0, s_1)$ under the embedding $C \times_k \operatorname{Spec} k(t_0) \xrightarrow{\text{id}, t_0} C \times \mathbb{A}^1$ gives a degree-$d$ polynomial in $s_0$ and $s_1$. A point $p \in C$ is a zero of this pullback section if and only if $(1-t_0) F(X,Y) + t_0F'(X,Y)$ vanishes at $[X:Y] = [s_0(p), s_1(p)] = \pi(p)$. Equivalently, the vanishing locus of the section $(1-t_0) F(s_0, s_1) + t_0 F'(s_0, s_1)$ is the fiber of the vanishing locus of $(1-t_0)F(X,Y) + t_0 F(X,Y)$ under $\pi$. The vanishing locus of $(1-t_0) F(X,Y) + t_0F'(X,Y)$ must be a zero dimensional subscheme of $\mathbb{P}^1_L$ because $(1-t_0) F(X,Y) + t_0 F'(X,Y)$ is a monic homogeneous polynomial of degree $d$. Moreover, since $\pi$ is a finite morphism, the vanishing locus of the section $(1-t_0) F(s_0, s_1) + t_0 F'(s_0,s_1)$ must be zero dimensional as well and hence the section has isolated zeros. Therefore, $F(s_0,s_1)$ and $F'(s_0, s_1)$ are connected by sections with isolated zeros, so Theorem \ref{thm:eulernumberconnectedsection} concludes that their Euler numbers are equal.

\end{proof} 
\end{proposition}

\begin{corollary} \label{euler deg independent}
Assuming that $\pi^*\mathscr{O}_{\mathbb{P}^1_k}(\deg q)$ is relatively oriented, the Euler number $e(\pi, q)$ depends only on $\deg q$. 
\begin{proof}
Say that $q_1$ and $q_2$ are two closed points of $\mathbb{P}^1_k$ of equal degree $d$ over $k$. Let $F_1(X,Y)$ and $F_2(X,Y)$ be their respective irreducible (homogeneous) polynomials. By Proposition \ref{proposition:eulernumberpoly} above, 
$$e(\pi^* \mathscr{O}_{\mathbb{P}^1_k}(d), F(s_0,s_1)) = e(\pi^* \mathscr{O}_{\mathbb{P}^1_k}(d), F'(s_0,s_1)).$$
Thus,
$$e(\pi,q_1) = e(\pi,q_2).$$
\end{proof}
\end{corollary}

Based on Corollary \ref{euler deg independent} above, we can define the global Euler $\A^1$-degree by letting $q$ be $k$-rational.

\begin{definition} \label{naiveEulerequiv}
Let $\pi: C \to \Proj^1_k$ be a finite morphism of smooth curves over $k$. Let $\pi = [s_0: s_1]$ where $s_0, s_1$ are global sections of $\pi^* \Oh_{\Proj^1_k}(1)$ with no common zeroes. Assume that $\pi^* \Oh_{\Proj^1_k}(1)$ is relatively oriented. Let $q \in \Proj^1_k(k)$ be a $k$-rational point, $F_q(X,Y)$ the monic minimal polynomial of $q$, and $Z$ the zero locus of $F_q(s_0,s_1)$, which is considered as a section of $\pi^* \Oh_{\Proj^1_k}(1)$. We define the global Euler $\A^1$-degree, or the Euler number, of $\pi: C \to \mathbb{P}^1_k$ by
\begin{equation*}
    \EAdeg \pi := e(\pi, q) = \sum_{p \in Z} \ind_{p} F_q(s_0, s_1).
\end{equation*}
\end{definition}

\begin{definition}
Let $\pi: E \to X$ be a relatively oriented line bundle on a smooth curve $X$ over $k$, and let $\sigma$ be a section with isolated zeros. We say that a zero $p$ of $\sigma$ has ramification index $e$ if, for a trivializing open neighborhood $U$ around $p$ in $X$ and a trivialization $E|_U \cong \mathscr{O}_U$, the function corresponding to $\sigma|_U$ yields a map $U \to \mathbb{A}^1$ of curves of ramification index $e$. Note that this definition does not depend on $U$ or the trivialization.
\end{definition}

We now show that the local index of a section of a relatively oriented line bundle at a doubly ramified point is a multiple of the hyperbolic element $\mathbb{H} = \langle 1 \rangle + \langle -1 \rangle \in \text{GW}(k)$. In Section \ref{sec: covering maps}, we will use the fact that the covering maps from many modular curves to $X(1)$ 
have fibers at $j = 1728$ consisting entirely of doubly ramified points to conclude that the Euler numbers of these maps are integer multiples of the hyperbolic element.

\begin{proposition}\label{euler double ramify}
Let $\pi: E \rightarrow X$ be a relatively oriented line bundle on a smooth curve $X$ over $k$, let $\sigma$ be a section with isolated zeros, let $Z$ be the zero locus of $\sigma$, and let $p \in Z$.

If $p \in Z$ is a doubly ramified point for $\pi$ such that $k(p)$ is separable over $k$, then $$ \ind_p \sigma = [k(p):k]( \langle 1 \rangle + \langle - 1 \rangle).$$

\begin{proof}
It suffices to show this in the case where $k = k(p)$ because 

$$\ind_p \sigma = \Tr_{k(p)/k} \ind_{p_{k(p)}} \sigma_{k(p)},$$

by Proposition \ref{local_index_base_change}, and because Lemma \ref{TraceDoubleRamify} tells us that 

$$\Tr_{k(p)/k} (\langle 1 \rangle + \langle -1 \rangle) = [k(p):k](\langle 1 \rangle + \langle -1 \rangle).$$

Now assume that $k = k(p)$. Let $\phi: U \rightarrow \Spec(k[x])$ be Nisnevich coordinates around $p$, which exist by Proposition \ref{curvenis}. Via $\phi$, the local ring $\mathscr{O}_{Z,p}$ can be presented as $k[x]_{q}/(f)$ where $q = \phi(p)$ for some polynomial $f \in k[x]$ \cite[Section 4, the discussion after Lemma 24 leading up to Lemma 27]{kwcubic}. 

We show that $\Omega_{p/Z} \cong \Omega_{(k[x]_q/(f))/k}$. Note that there is an closed embedding $p \hookrightarrow \Spec(\mathscr{O}_{Z,p})$. The composition of this closed embedding with the isomorphism $\mathscr{O}_{Z,p} \cong \Spec(k[x]_q/(f))$ is the same as the closed embedding $q \hookrightarrow \Spec(k[x]_q/(f))$. In particular, 
$$\Omega_{p/\Spec(\mathscr{O}_{Z,p})} \cong \Omega_{q/\Spec(k[x]_q/(f))} \cong \Omega_{(k[x]_q/(f))/k}.$$
On the other hand, $Z$ consists of finitely many closed points \ref{isozero} and is hence finite over $\Spec(k)$. Therefore, $\Spec(\mathscr{O}_{Z,p})$ is an open subscheme of $Z$, so 
$$\Omega_{\Spec(\mathscr{O}_{Z,p})/Z} = 0$$
and hence
$$\Omega_{p/\Spec(\mathscr{O}_{Z,p})} \cong \Omega_{p/Z}.$$
Thus, $\Omega_{p/Z} \cong \Omega_{(k[x]_q/(f))/k}$ as desired.

Let $g(x) \in k[x]$ be the monic generator of the prime ideal $q$ of $k[x]$. In particular, $g$ is linear because $p$, and hence $q$, has residue field $k$. Moreover, since $p$ is doubly ramified in $Z$ and since $\Omega_{p/Z} \cong \Omega_{(k[x]_q/(f))/k}$, the point $q$ is doubly ramified in $\Spec(k[x]_q/(f))$. Therefore, $f$ is divisible by $g$ exactly twice and hence the local ring $k[x]_q/(f)$ can be presented as $k[x]_q/(g^2)$.

By \cite[Section 3]{SS75}, cf. \cite[Section 4, the discussion after Lemma 27]{kwcubic}, the presentation $k[x]_q/(g^2) \cong \mathscr{O}_{Z,p}$ of the complete intersection $k$-algebra $\mathscr{O}_{Z,p}$ determines a canonical isomorphism
$$\Hom_k(\mathscr{O}_{Z,p},f) \cong \mathscr{O}_{Z,p}$$
of $\mathscr{O}_{Z,p}$-modules. We take $\eta \in \Hom_k(\mathscr{O}_{Z,p},f)$ to correspond to $1 \in \mathscr{O}_{Z,p}$. In Definition \ref{def:localind}, $\ind_p \sigma$ is defined as the Grothendieck-Witt group element induced by the symmetric bilinear form $\beta: \mathscr{O}_{Z,p} \times \mathscr{O}_{Z,p} \rightarrow k, (x,y) \mapsto \eta(xy)$. When $p$ is a $k$-point, one may replace $\eta$ with any $k$-linear homomorphism taking the distinguished socle element to $1$. In this case, $g \in k[x]_q/(g^2)$ is the distinguished socle element \cite[Example 32]{kwcubic}. For instance, we may choose $\eta$ to be the $k$-linear map $k[x]_q/(g^2) \rightarrow k$ taking $1$ to $0$ and $g$ to $1$. Giving $k[x]_q/(g^2)$ the $k$-basis $(1,g)$, the bilinear form $\beta$ has matrix representation
$$\begin{pmatrix} 0 & 1 \\ 1 & 0 \end{pmatrix}$$
which reduces to the element $\langle 1 \rangle + \langle -1 \rangle$ in $\GW(k)$ by Lemma \ref{bezoutlemma}.
\end{proof}
\end{proposition}

We note that \cite[Example 32]{kwcubic} and \cite[Main Theorem, Proposition 14]{kwEKL} imply that if $C = \Proj^1_k$, then the Euler number of $\pi$, as defined above, is equal to Morel's $\A^1$-Brouwer degree of $\pi$. Its value, as an element in $\GW(k)$, can be directly computed from the EKL class of a non-degenerate bilinear form $\beta$ induced from $\pi$ considered as a morphism $\A^1_k \to \A^1_k$.

Suppose that $\pi: C \to \mathbb{P}^1$ is a morphism such that $\pi^* \Oh_{\mathbb{P}^1_k}(1)$ is relatively oriented, and there exists a $k$-rational point $q \in \mathbb{P}^1(k)$ at which all of its fibers doubly ramify. There are two separate ways to compute the global Euler $\A^1$ degree of $\pi$. The first method immediately follows from Proposition \ref{euler double ramify}.

\begin{proposition} \label{A1 degree double ramify}
    Let $C/k$ be a smooth projective curve. Suppose $\pi: C \to \mathbb{P}^1_k$ is a morphism such that $\pi^* \Oh_{\mathbb{P}^1}(1)$ is relatively oriented, and there exists a $k$-rational point $q \in \mathbb{P}^1_k(k)$ whose fiber points all doubly ramify and all have separable residue field over $k$. Then the global Euler $\A^1$ degree of $\pi$ is given by
    \begin{equation*}
        \EAdeg \pi = \frac{\deg \pi}{2} (\langle 1 \rangle + \langle -1 \rangle).
    \end{equation*}
\end{proposition}
\begin{proof}
    By definition of the global Euler $\A^1$ degree,
    \begin{equation*}
        \EAdeg \pi = \sum_{p \in Z} \ind_{p} F_q(s_0, s_1).
    \end{equation*}
    where $F_q(S_0,s_1)$ is a global section of the line bundle $\pi^* \Oh_{\mathbb{P}^1}(1)$. Since every fibral point of $\pi$ at $q$ is doubly ramified, the reduction of this fiber is of degree $\frac{1}{2} \deg \pi$. By Proposition \ref{euler double ramify},
    \begin{equation*}
        \EAdeg \pi = \sum_{p \in Z} [k(p): k] (\langle 1 \rangle + \langle -1 \rangle) = \frac{\deg \pi}{2} (\langle 1 \rangle + \langle -1 \rangle).
    \end{equation*}
\end{proof}

\begin{remark} \label{SW21}
A consequence of Srinivasan and Wickelgren's recent result \cite[Proposition 19]{SW21} is that the global Euler $\A^1$-degree of $\pi$ is always equal to an integer multiple of the hyperbolic element, as long as $\pi^* \Oh_{\Proj^1}(1)$ is relatively oriented. Nevertheless, we include the results on global Euler $\A^1$ degrees to make the paper self-contained and to compare the two different notions of global $\A^1$ degrees of morphisms, as will appear in subsequent sections of the manuscript.
\end{remark}

\subsection{The Euler number of a finite map into \texorpdfstring{$\A^n$}{An}.} \label{subsec: euler number into affine}

In Section \ref{subsec: explicit method} below, we roughly reduce computations of local indices for finite maps from a general smooth affine scheme $Y$ to an affine space to computations of local indices for finite maps from an affine space to itself. For the latter type of map, the local indices are equivalent to the local $\A^1$-degrees \cite[Example 32]{kwcubic}, which can be computed by the methods in \cite{kwEKL} that we describe above in Definition \ref{def:ekl} and Theorem \ref{localEKL}.

In this section, we define the Euler number $e(\pi, q)$ for a finite $k$-map $\pi: Y \rightarrow \A^n_k$ at a closed point $q \in \A^n_k$. This definition is similar to that for finite maps into $\Proj^1_k$ in Section \ref{sec: Relative orientability} above. The precise conditions under which we define $e(\pi, q)$ are also similar to those in Section \ref{sec: Relative orientability}.

Given a scheme $Y$, recall that a map $\pi: Y \rightarrow \A^n_k$ corresponds to a choice of a global section of $\Oh_Y^n$, which in turn corresponds to a choice of $n$ global sections of $\Oh_Y$. In particular, the components of the map are the $n$ global sections of $\Oh_Y$.

Given a closed point $q \in \A^n_k$, let $f_1(x_1),\ldots,f_n(x_n) \in k[x]$ be nonzero minimal (monic) polynomials for the coordinates $q_1,\ldots,q_n$ of $q$ respectively. If $q_i = 0$, then we let $f_i(x_i) = x_i$. The map
\begin{align*}
    \A^n &\rightarrow \A^n \\
    (t_1,\ldots,t_n) &\mapsto (f_1(t_1),\ldots,f_n(t_n))
\end{align*}
corresponds to a section of $\Oh_{\A^n}^n$ and this section has zeroes precisely at $q$. Pulling back this section via $\pi$ yields a section of $\Oh_Y^n$ whose zeroes are precisely at $\pi^{-1}(q)$. Note that the points of $\pi^{-1}(q)$ are isolated because $\pi$ is finite.

Assuming that $Y$ is smooth of dimension $n$ and that $\Oh_Y^n$ is relatively oriented, we define the Euler number $e(\pi, q)$ of $\pi$ at $q$. Note that $\Oh_Y^n$ is relatively orientable exactly when $\bigwedge^{\text{top}} TY$ is orientable because $\bigwedge^{\text{top}} \Oh_Y^n = \Oh_Y$.

\begin{definition} \label{def:affine euler number}
Let $Y$ be a smooth scheme of dimension $n$ with a finite map $\pi: Y \rightarrow \A^n$. Assume that $\Oh_Y^n$ is relatively oriented. Let $q \in \A^n$ be a closed point. For $1 \leq i \leq n$, let $f_i(x_i)$ be the (monic) minimal polynomial of the coordinate $q_i$ of $q$. The map $(f_1(x_1),\ldots,f_n(x_n)): \A^n \rightarrow \A^n$ corresponds to a section $\sigma_q$ of $\Oh_{\A^n}^n$. Define the Euler number $e(\pi, q)$ by
$$e(\pi,q) := e(\Oh^n_Y, \pi^* \sigma_q).$$
\end{definition}

In the case that $n = 1$, Definitions \ref{naiveEulerequiv} and \ref{def:affine euler number} coincide. 

\begin{proposition} \label{curve_euler_number_2}
Given a smooth projective curve $C$ over any field $k$, let $\pi = [s_0:s_1]: C \to \Proj^1_k$ be a finite morphism. Let $Y = \pi^{-1}(\A^1_k)$, which is an open affine subscheme of $C$. Suppose that $\pi^* \Oh_{\Proj^1_k}(1)$ is relatively oriented. Pick a $k$-rational point $q \in \A^1_k \subset \Proj^1_k$. Let $F_q(X, Y)$ be the monic minimal polynomial of $q$. Then for any $p \in \pi^{-1}(q) = \pi|_Y^{-1}(q)$,
\begin{equation*}
    \ind_p F_q(s_0,s_1) = \ind_p F_q(s_0/s_1, 1)
\end{equation*}
where the left hand side is the local index of the Euler number $e(\pi, q)$ and the right hand side is the local index of the Euler number $e(\pi|_Y, q)$ defined by restricting the relative orientation on $\pi^*\mathscr{O}_{\mathbb{P}^1_k}(1)$ to $Y$. In particular,
\begin{equation*}
    e(\pi, q) = e(\pi|_Y, q).
\end{equation*}
\end{proposition}
\begin{proof}
Because $C$ is a smooth projective curve over $k$, both $C$ and $Y$ are local complete intersections over $k$. Because $Y$ is an affine local complete intersection, $Y$ is a set theoretic complete intersection \cite[Corollary 5]{ku78}. The relative orientation
\begin{equation*}
    \text{Hom}(TC, \pi^* \Oh_{\Proj^1_k}(1)) \cong L^{\otimes 2}
\end{equation*}
on $\pi^* \Oh_{\Proj^1_k}(1)$ restricts to the relative orientation
\begin{equation*}
    \text{Hom}(TY, \Oh_Y) \cong \text{Hom}(TY, \pi|_Y^* \Oh_{\A^1}) \cong \text{Hom}(TC|_{Y}, \pi^* \Oh_{\Proj^1_k}(1)|_{Y}) \cong L|_Y^{\otimes 2}.
\end{equation*}
on the structure sheaf $\Oh_Y$.

Let $\phi: U \to \A^1_k$ be a Nisnevich local coordinate of $p \in C$. Then the restriction $\phi|_Y: U \cap Y \to \A^1_k$ is also a Nisnevich local coordinate at $p \in Y \subset C$. We can choose a compatible trivialization of $\pi^* \Oh_{\Proj^1_k}(1)|_{U \cap Y} = \Oh_Y|_{U \cap Y}$ by sufficiently shrinking $U \cap Y$. Because both $\ind_p F_q(s_0, s_1)$ and $\ind_p F_q(s_0/s_1,1)$ are induced from the same choice of the Nisnevich local coordinates and compatible trivializations, we obtain
\begin{equation*}
    \ind_p F_q(s_0,s_1) = \ind_p F_q(s_0/s_1, 1).
\end{equation*}
A summation of local indices at every element of $\pi^{-1}(q)$ shows
\begin{equation*}
    e(\pi, q) \overset{def}{=} \sum_{p \in \pi^{-1}(q)} \ind_p F_q(s_0,s_1) = \sum_{p \in \pi|_Y^{-1}(q)} \ind_p F_q(s_0/s_1,1) \overset{def}{=} e(\pi|_Y, q)
\end{equation*}
\end{proof}

\subsection{An explicit method to compute local indices and global \texorpdfstring{$\A^1$}{A1} degrees}
\label{subsec: explicit method}

Let $\pi: C \to \mathbb{P}^1_k$ be a finite morphism of smooth projective curves such that $\pi^* \Oh_{\Proj^1_k}(1)$ is relatively oriented. Let $Y$ be an open affine subscheme of $C$ which is a trivializing open subscheme of $C$ for $\pi^* \Oh_{\Proj^1_k}(1)$. In particular, we may choose $Y = \pi^{-1}(\A^1_k)$. 

Using the Euler number of the restriction $\pi|_Y: Y \to \A^1_k$, we present a method to explicitly compute the local indices of $e(\pi,q)$ for any $k$-rational point $q \in \A^1_k \subset \mathbb{P}^1_k$. The following proposition proves how the local indices of the Euler number $e(\pi|_Y, q)$ can be explicitly computed up to a unit in $k$ under certain conditions.

\begin{proposition} \label{curve_euler_number}
Let $Y$ be a finite type smooth affine scheme of dimension $n$ such that $\Oh^n_Y$ is relatively oriented over $k$. Pick a closed embedding $i_Y: Y \to \A^r_k$ for some $r$ and let $\{x_1, x_2, \cdots, x_r\}$ be coordinates of $\A^r_k$. Suppose that $Y$ is a set theoretic complete intersection over $k$, i.e. the zero set of the collection of $r-n$ polynomials $\{F_1, F_2, \cdots, F_{r-n}\}$ defines $Y$. 

Let $\mathcal{F}: Y \to \A^n_k$ be a finite morphism defined by the polynomial map $(F_{r-n+1}, \cdots, F_r)$ over $k[x_1, \cdots, x_r]$. Let $F = (F_1, F_2, \cdots, F_r): \A^r_k \to \A^r_k$ be a finite morphism such that the diagram
\begin{equation*}
    \begin{tikzcd}
        \A^r_k \arrow[r, "F"] & \A^r_k \\
        Y \arrow[r, "\mathcal{F}"] \arrow[u, "i_Y"] & \A^n_k \arrow[u, "i_{\A^n_k}"]
    \end{tikzcd}
\end{equation*}
commutes, where $i_{\A^n_k}: \A^n_k \to \A^r_k$ sends $q := (q_1, q_2, \cdots, q_n) \in \A^n_k$ to $(0, \cdots, 0, q_1, q_2, \cdots, q_n) \in \A^r_k$. Note that $\mathcal{F}$ and $F$ can be considered as global sections of $\Oh^n_Y$ and $\Oh^r_{\A^r_k}$ respectively. Moreover, given a point $q \in \A^n_k$, note that the fibers of the point $(0,\ldots,0, q) \in \A^r_k$ via $F$ are all in $Y$. Denote by $\{p_1, \cdots, p_d\}$ the elements of $\mathcal{F}^{-1}(q)$. Assume that for all $1 \leq i \leq d$'s, the fields $k(p_i)$ are separable extensions of $k$.

Pick a $k$-rational point $q \in \A^n_k(k)$. Let $\mathcal{G}$ and $G$ be global sections of $\Oh^n_{\A^n_k}$ and $\Oh^r_{\A^r_k}$ given by the monic minimal polynomials of the coordinates of $q$ and $i_{\A^n_k}(q)$ respectively, cf. Section \ref{subsec: euler number into affine}. In particular, the zero sets of $\mathcal{G}$ and $G$ are exactly $q$ and $i_{\A^n_k}(q)$ respectively. Furthermore, the zero sets of $\mathcal{F}^* \mathcal{G}$ and $F^* G$ are $\mathcal{F}^{-1}(q)$ and $F^{-1}(i_{\A^n_k}(q))$ respectively. Then for each closed point $p_i \in \mathcal{F}^{-1}(q)$ there exists a unit $\alpha_{p_i} \in k(p_i)^\times$ such that
\begin{equation*}
    \ind_{(p_i)_{k(p_i)}} \mathcal{F}^* \mathcal{G} = \langle \alpha_{p_i} \rangle \bullet \ind_{(i_Y(p_i))_{k(p_i)}} F^* G.
\end{equation*}
In particular, if $\alpha_{p_i} \in k^\times$, then 
\begin{equation*}
    \ind_{p_i} \mathcal{F}^* \mathcal{G} = \langle \alpha_{p_i} \rangle \bullet \ind_{i_{Y}(p_i)} F^* G.
\end{equation*}
Assuming that there exists some $\alpha \in k^\times$ such that $\alpha_{p_i} \in k^\times$ for all $i$'s and $\frac{\alpha}{\alpha_{p_i}}$ is a square in $k^\times$ for every $i$, 
\begin{equation*}
    e(\mathcal{F}, q) = \langle \alpha \rangle \bullet e(F, i_{\A^n_k}(q)).
\end{equation*}
\end{proposition}

\begin{proof}
The proposition can be considered as a corollary of various lemmas and propositions stated in \cite[Section 4]{kwcubic}.
Let $\text{T}Y \to Y$ be the tangent bundle of $Y$. Because $\Oh^n_Y$ is relatively oriented, we have a line bundle $L$ on $Y$ along with an isomoprhism
\begin{align*}
    \text{Hom}(\bigwedge\nolimits^n \text{T}Y, \bigwedge\nolimits^n \Oh^n_Y) \cong L^{\otimes 2}.
\end{align*}
Moreover, since all line bundles on $\mathbb{A}^r_k$ are trivial, we obtain a relative orientation of $\Oh^r_{\A^r_k}$ by restriction:
\begin{align*}
    \text{Hom} (\bigwedge\nolimits^r \text{T}\A^r_k, \bigwedge\nolimits^r \Oh^r_{\A^r_k}) \cong \Oh_{\A^r_k}^{\otimes 2}.
\end{align*}
Let $\mathcal{Z} \subset Y$ be the zero locus of $\mathcal{F}^* \mathcal{G}$, and $Z \subset \A^r$ be the zero locus of $F^* G$. In particular, $\mathcal{Z} = \mathcal{F}^{-1}(q)$ and $Z = F^{-1}(i_{\A^n_k}(q))$. Furthermore, the morphism $i_Y|_{\mathcal{Z}}: \mathcal{Z} \to Z$ is an isomorphism.

Pick a fiber point $p \in \mathcal{Z}$. Let $\phi_Y: \mathcal{U} \to \A^n_k = \Spec k[y_1, y_2, \cdots, y_n]$ be a choice of Nisnevich local coordinates of $Y$ at $p$, and let $\phi_{\A^r_k}: U \to \A^r_k = \Spec k[\chi_1, \chi_2, \cdots, \chi_r]$ be a choice of Nisnevich local coordinates of $\A^r_k$ at $i_Y(p)$. By sufficiently shrinking the open subschemes $\mathcal{U}$ and $U$, we can choose compatible trivializations of $\Oh^n_Y$ and $\Oh^r_{\A^r_k}$. The proof of \cite[Lemma 24]{kwcubic} implies that the following two diagrams are commutative:
\begin{equation*}
\begin{tikzcd}
    & k[\chi_1, \chi_2, \cdots, \chi_r] \arrow[d, twoheadrightarrow] \arrow[dl] \\
    \Oh_{\A^r_k, i_Y(p)} \arrow[r, twoheadrightarrow] & \Oh_{Z,i_Y(p)}
\end{tikzcd}
\end{equation*}
\begin{equation*}
\begin{tikzcd}
    & k[y_1, y_2, \cdots, y_n] \arrow[d, twoheadrightarrow] \arrow[dl] \\
    \Oh_{Y, p} \arrow[r, twoheadrightarrow] & \Oh_{\mathcal{Z}, p}
\end{tikzcd}
\end{equation*}

By exactness of localization, the surjection 
$$i_Y^*: k[x_1, x_2, \cdots, x_r] \to k[x_1, x_2, \cdots, x_r] / (F_1, F_2, \cdots, F_r)$$ 
associated to the closed embedding $i_Y:Y \to \A^r_k$, induces
\begin{align*}
    \Oh_{Z,i_Y(p)} & \cong \Oh_{\A^r_k, i_Y(p)} / (F_1, F_2, \cdots, F_r) \\
    & \cong k[x_1, x_2, \cdots, x_r]_{i_Y(p)} / (F_1, F_2, \cdots, F_r) \\
    & \cong \left( k[x_1, x_2, \cdots, x_r]/(F_1, \cdots, F_{r-n}) \right)_{i_Y(p)} / (F_{r-n+1}, F_{r-n+2}, \cdots, F_r) \\
    & \cong \Oh_{Y, p} / \left( (F_{r-n+1}, F_{r-n+2}, \cdots, F_r) \right) \\
    & \cong \Oh_{\mathcal{Z}, p}.
\end{align*}

Given a closed point $p$ such that the field extension $k(p)/k$ is separable, the canonical point $p_{k(p)}$ of $Y_{k(p)}$ is a $k(p)$-rational point of $Y_{k(p)}$. By the discussion in Definition \ref{def:localind}, the local index $\ind_{p_{k(p)}} \mathcal{F}^* G$ is the equivalence class of the symmetric bilinear form $\tilde{\beta}: \Oh_{\mathcal{Z}_{k(p)}, p_{k(p)}} \times \Oh_{\mathcal{Z}_{k(p)}, p_{k(p)}} \to k(p), \quad (x,y) \mapsto \tilde{\eta}(xy)$ where $\tilde{\eta}: \Oh_{\mathcal{Z}_{k(p)}, p_{k(p)}} \to k(p)$ is any $k(p)$-linear map sending the distinguished socle element of $\Oh_{\mathcal{Z}_{k(p)}, p_{k(p)}}$ to $1$. Similarly, $\ind_{(i_Y(p))_{k(p)}}F^* G$ is the equivalence class of the symmetric bilinear form $\beta: \Oh_{Z_{k(p)}, (i_Y(p))_{k(p)}} \times \Oh_{Z_{k(p)}, (i_Y(p))_{k(p)}} \to k(p), \quad (x,y) \mapsto \eta(xy)$ where $\eta: \Oh_{Z_{k(p)}, (i_Y(p))_{k(p)}} \to k(p)$ sends the distinguished socle element of $\Oh_{Z_{k(p)}, (i_Y(p))_{k(p)}}$ to $1$. We recall that \cite[Lemma 27]{kwcubic} shows that both $\Oh_{\mathcal{Z}_{k(p)}, p_{k(p)}}$ and $\Oh_{Z_{k(p)},(i_Y(p))_{k(p)}}$ are finite complete intersections. The isomorphism $i_Y^*: \Oh_{Z_{k(p)},(i_Y(p))_{k(p)}} \to \Oh_{\mathcal{Z}_{k(p)}, p_{k(p)}}$ thus sends the socle element $s$ to $\alpha_p \tilde{s}$ for some $\alpha_p \in k(p)^\times$. Therefore,
\begin{equation*}
    \ind_{p_{k(p)}} \mathcal{F}^* \mathcal{G} = \langle \alpha_{p} \rangle \bullet \ind_{(i_{Y}(p))_{k(p)}} F^* G.
\end{equation*}

If we further assume that $\alpha_p \in k^\times$, then Proposition \ref{local_index_base_change} implies
\begin{align*}
    \ind_p \mathcal{F}^* \mathcal{G} &= \Tr_{k(p)/k} \ind_{p_{k(p)}} \mathcal{F}^* \mathcal{G} \\
    &=  \Tr_{k(p)/k} \left( \langle \alpha_{p} \rangle \bullet \ind_{(i_{Y}(p))_{k(p)}} F^* G \right) \\
    &= \langle \alpha_p \rangle \bullet \Tr_{k(p)/k} \ind_{(i_{Y}(p))_{k(p)}} F^* G \\
    &= \langle \alpha_p \rangle \bullet \ind_{i_{Y}(p)} F^* G.
\end{align*}

Now let $\mathcal{F}^{-1}(q) = \{p_1, p_2, \cdots, p_d\}$. Further suppose that for all $p_i$'s the associated units $\alpha_{p_i}$ are elements of $k^\times$ and are all equal to $\alpha$ up to multiplication by a square in $k$. We show that
$$e(\mathcal{F},q) = \langle \alpha \rangle \bullet e(F, i_{\A^n_k}(q)).$$
Recall that $q \in  \A^n_k$. Choose $g_{r-n+1}(x_{r-n+1}),\ldots,g_r(x_r)$ to be nonzero (monic) minimal polynomials of the coordinates of $q$ similarly as in Section \ref{subsec: euler number into affine}. These $n$ polynomials together correspond to a section $\sigma_q$ of $\Oh^n_{\A^n_k}$. Let $\mathcal{G} = \mathcal{F}^* \sigma_q$. Similarly, the $n$ polynomials $x_1,\ldots,x_{r-n},g_{r-n+1}(x_{r-n+1}),\ldots,g_r(x_r)$ correspond to a section $\sigma'_q$ of $\Oh^r_{\A^r_k}$. Let $G = F^* \sigma'_q$. Note that $\mathcal{G}$ and $G$ are global sections of $\Oh_Y^n$ and $\Oh^r_{\A^r_k}$ whose zero sets are $\mathcal{F}^{-1}(q)$ and $F^{-1}(q)$ respectively.

Furthermore, choose $\mathcal{G}$ and $G$ to be the global sections of $\Oh^n_Y$ and $\Oh^r_{\A^r_k}$ 
respectively Then
\begin{align*}
    e(\mathcal{F}, q) 
    &\overset{\mathrm{def}}{=} e(\Oh^n_Y, \mathcal{F}^* \mathcal{G}) \\
    &\overset{\mathrm{def}}{=} \sum_{p_i \in \mathcal{F}^{-1}(q)} \ind_{p_i} \mathcal{F}^* \mathcal{G} \\
    &= \sum_{p_i \in \mathcal{F}^{-1}(q)} \langle \alpha \rangle \bullet \ind_{i_Y(p_i)} F^* G \\
    &=  \sum_{i_Y(p_i) \in F^{-1}(i_{\A^n_k}(q))} \langle \alpha \rangle \bullet \ind_{i_Y(p_i)} F^* G \\
    &\overset{\mathrm{def}}{=} \langle \alpha \rangle \bullet e(\Oh^r_{\A^r_k}, F^* G) \\
    &\overset{\mathrm{def}}{=} \langle \alpha \rangle \bullet e(F, i_{\A^n_k}(q)).
\end{align*}

\end{proof}

Combining Proposition \ref{curve_euler_number_2} and \ref{curve_euler_number}, we obtain:
\begin{corollary} \label{cor:curve_euler_number}
    Let $C$ be a smooth projective curve over any field $k$. Let $\pi = [s_0:s_1]: C \to \Proj^1_k$ be a finite morphism such that $\pi^* \Oh_{\Proj^1_k}(1)$ is relatively oriented. Let $Y = \pi^{-1}(\A^1_k)$ with a choice of a closed embedding $i_Y: Y \to \A^r_k$ in which $Y$ is a set theoretic complete intersection, i.e. a set of $r-1$ polynomials $\{F_1, F_2, \cdots, F_{r-1}\}$ defines $Y$. Let $\{x_1,x_2,\cdots,x_r\}$ be coordinates of $\A^r_k$.
    
    Pick a $k$-rational point $q \in \A^1_k \subset \Proj^1_k$. Let $F_q(X, Y)$ be the monic minimal polynomial of $q$, where $X, Y$ are the coordinates of $\Proj^1_k$ and $\A^1_k$ has the coordinates $\frac{X}{Y}$. Define $F: \A^r_k \to \A^r_k$ to be the polynomial map $F = (F_1, \cdots, F_{r-1}, \pi|_Y)$. In particular, the following diagram commutes:
    \begin{equation*}
        \begin{tikzcd}
                    \A^r_k \arrow[r, "F"] & \A^r_k \\
        Y \arrow[r, "\pi|_Y"] \arrow[u, "i_Y"] \arrow[d, "i"] & \A^1_k \arrow[u, "i_{\A^n_k}"] \arrow[d, "i"] \\
        C \arrow[r, "\pi"] & \Proj^1_k
        \end{tikzcd}.
    \end{equation*}
    Then for each closed point $p \in \pi^{-1}(q)$, there exists a unit $\alpha_p \in k(p)^\times$ such that
    \begin{equation*}
        \ind_{(i(p))_{k(p)}} F_q(s_0,s_1) = \langle \alpha_p \rangle \bullet \ind_{(i_Y(p))_{k(p)}} F^* (x_1,x_2,\cdots,x_{r-1},F_q(x_r,1)).
    \end{equation*}
    If $\alpha_p \in k^\times$ for all the closed points $p \in \pi^{-1}(q)$, then
    \begin{equation*}
        \ind_{i(p)} F_q(s_0,s_1) = \langle \alpha_p \rangle \bullet \ind_{i_Y(p)} F^* (x_1,x_2,\cdots,x_{r-1},F_q(x_r,1)).
    \end{equation*}
\end{corollary}

Using Proposition \ref{euler double ramify} and Corollary \ref{cor:curve_euler_number}, we are able to obtain under which conditions can we guarantee the equality between global Euler $\A^1$ degrees and the global $\A^1$ degrees of morphisms $\pi: C \to \Proj^1_k$.

\begin{theorem} \label{thm:two_notions_agree}
Let $C$ be a smooth projective curve over a field $k$. Suppose $\pi: C \to \Proj^1_k$ is a finite morphism which satisfies the following conditions:
\begin{enumerate}[label=(\alph*)]
    \item \label{thm:tna_double_ramify} there exists a $k$-rational point $q \in \Proj^1_k$ such that all closed points in $\pi^{-1}(q)$ doubly ramify and have separable residue field over $k$, and
    \item \label{thm:tna_complete_intersection} there exists a positive number $r > 0$ such that the open subscheme $Y := \pi^{-1}(\A^1_k)$ is a set theoretic complete intersection in $\A_k^r$.
\end{enumerate}
If $\pi^* \Oh_{\mathbb{P}^1}(1)$ is relatively oriented, then the global Euler $\A^1$ degree of $\pi$ is 
\begin{equation*}
    \EAdeg \pi = \frac{\deg \pi}{2} \left( \langle 1 \rangle + \langle -1 \rangle \right).
\end{equation*}
If $\pi$ is relatively oriented, i.e. the pullback of the line bundle $\pi^* \Oh_{\mathbb{P}^1}(2) = \pi^* T\mathbb{P}^1$ is relatively oriented, then the global $\A^1$ degree of $\pi$ is 
\begin{equation*}
    \Adeg \pi = \frac{\deg \pi}{2} \left( \langle 1 \rangle + \langle -1 \rangle \right).
\end{equation*}
In particular, if both line bundles are relatively oriented, then the two notions of global $\A^1$ degrees of $\pi$ are equal to the integer multiple of the hyperbolic element in $\GW(k)$.
\end{theorem}
\begin{proof}
Proposition \ref{A1 degree double ramify} implies that if $\pi^* \Oh_{\mathbb{P}^1}(1)$ is relatively oriented, then the global Euler $\A^1$ degree of $\pi$ is an integer multiple of the hyperbolic element.
\begin{equation*}
    \EAdeg \pi = \sum_{p \in \pi^{-1}(q)} [k(p):k] \left( \langle 1 \rangle + \langle -1 \rangle \right) = \frac{\deg \pi}{2} \left( \langle 1 \rangle + \langle -1 \rangle \right).
\end{equation*}

To compute the global $\A^1$ degree of the morphism $\pi$, we use Definition \ref{sum_local_A1} to express it as a sum of local $\A^1$ degrees of $\pi$ given a $k$-rational point $q \in \mathbb{P}^1$:
\begin{equation*}
    \Adeg \pi = \sum_{p \in \pi^{-1}(q)} \Adeg_p \pi.
\end{equation*}
The notation $\Adeg_p \pi$ denotes the local $\A^1$ degree of $\pi$ at each of the closed points $p \in \pi^{-1}(q)$. The morphism $\pi$ induces the map of tangent bundles $T\pi: TC \to T\mathbb{P}^1$, which can be considered as an element of $\text{Hom}(TC, \pi^* T\mathbb{P}^1)(C)$. Choose a suitable Nisnevich local coordinates $\phi: \mathcal{U} \to \A^1_k$ which induces a trivialization of the tangent bundle $TC$. With respect to this trivialization, Theorem \ref{localEKL} implies that the local $\A^1$ degree at $p$ can be identified with the trace of the EKL class over the field extension $k(p)/k$:
\begin{equation*}
    \Adeg_p \pi = \Tr_{k(p)/k} \; w_0(T\pi(p)).
\end{equation*}
The EKL class $w_0(T\pi(p))$ is induced from the distinguished socle element of the local ring $\Oh_{\mathcal{Z},p_{k(p)}}$, where $Z$ is the zero locus of $\pi$. 

Using condition \ref{thm:tna_complete_intersection} of the theorem, we can find a set of $r-1$ polynomials $\{F_1, F_2, \cdots, F_{r-1}\}$ which cuts out $Y$ as a closed subscheme of $\A^r_k$. Let $q \in \A^1_k \subset \Proj^1_k$ be a $k$-rational point. As before, denote by $F_q(X, Y)$ the monic minimal polynomial of $q$, with $X, Y$ the coordinates of $\Proj^1_k$ and $\frac{X}{Y}$ the coordinate of $\A^1_k$ and write $\pi = [s_0:s_1]$. 

Let $F: \A^r_k \to \A^r_k$ be the polynomial map $F = (F_1, \cdots, F_{r-1}, \pi|_Y)$ so that the diagram from the statement of Corollary \ref{cor:curve_euler_number} commutes:
    \begin{equation*}
        \begin{tikzcd}
                    \A^r_k \arrow[r, "F"] & \A^r_k \\
        Y \arrow[r, "\pi|_Y"] \arrow[u, "i_Y"] \arrow[d, "i"] & \A^1_k \arrow[u, "i_{\A^n_k}"] \arrow[d, "i"] \\
        C \arrow[r, "\pi"] & \Proj^1_k
        \end{tikzcd}.
    \end{equation*}
Corollary \ref{cor:curve_euler_number} then shows that there exists a unit $\beta_p \in k(p)^\times$ such that the embedding $i_Y: Y \to \A^r_k$ induces
\begin{align*}
    \Adeg_p \pi &= \Tr_{k(p)/k} \; w_0(T\pi(p)) \\
    &= \Tr_{k(p)/k} \left( \langle \beta_p \rangle \bullet \ind_{(i_Y(p))_{k(p)}} F^* (x_1,\cdots,x_{r-1},F_q(x_r,1)) \right).
\end{align*}

Condition \ref{thm:tna_double_ramify} implies that for every such $p \in \pi^{-1}(q)$,
\begin{align*}
    \ind_{(i_Y(p))_{k(p)}} F^* (x_1,\cdots,x_{r-1},F_q(x_r,1)) = \langle 1 \rangle + \langle -1 \rangle.
\end{align*}
Because the hyperbolic element is invariant with respect to multiplication by elements of form $\langle a \rangle$ for any $a \in k(p)^\times$, we have
\begin{align*}
    \Adeg_p \pi &= \Tr_{k(p)/k} \left( \langle \beta_p \rangle \bullet \ind_{(i_Y(p))_{k(p)}} F^* (x_1,\cdots,x_{r-1},F_q(x_r,1)) \right) \\
    &= \Tr_{k(p)/k} \left( \langle 1 \rangle + \langle -1 \rangle \right) \\
    &= [k(p):k] \left( \langle 1 \rangle + \langle -1 \rangle \right).
\end{align*}
This implies
\begin{equation*}
    \Adeg \pi = \sum_{p \in \pi^{-1}(q)} \Adeg_p \pi = \sum_{p \in \pi^{-1}(q)} [k(p):k] \left( \langle 1 \rangle + \langle -1 \rangle \right) = \frac{\deg \pi}{2} \left( \langle 1 \rangle + \langle -1 \rangle \right).
\end{equation*}
Hence, both degrees are equal to integer multiples of the hyperbolic element of $\GW(k)$.
\end{proof}

\begin{example} \label{example:ellipticcurve}
Let $E$ be an elliptic curve with a minimal Weierstrass model given by $y^2z = x^3 - xz^2$ over any field $k$ of characteristic not equal to $2$ and $3$. Let $\pi: E \to \Proj^1_k$ be the double cover defined by the projection $[x:y:z] \mapsto [x:z]$. We demonstrate that
\begin{equation*}
    \EAdeg \pi = \langle 1 \rangle + \langle -1 \rangle = \Adeg \pi.
\end{equation*}
Recall that $\pi^* \Oh(1) \cong \Oh_E(2 \infty)$, cf. \cite[The explanation before Theorem IV.4.1]{hartshorne}. Moreover, the tangent bundle $TE$ is trivial, so $\Oh_E(2\infty)$ is relatively orientable over any field $k$:

$$\operatorname{Hom}(\det T E, \det \pi^* \Oh_{\mathbb{P}^1_k}(1)) \cong \operatorname{Hom}(\Oh_{E}, \Oh_{E}(2 \infty)) \cong \Oh_{E}(2 \infty) \cong \Oh_{E}(\infty)^{\otimes 2}.$$ 

The morphism $\pi$ is doubly ramified at the fiber of $0 \in \mathbb{P}^1_k(k)$. Hence, Proposition \ref{A1 degree double ramify} shows that
\begin{equation*}   
    \EAdeg \pi = \langle 1 \rangle + \langle -1 \rangle.
\end{equation*}

Let $\tilde{E}$ be the affine part of the elliptic curve $E$ at $z = 1$. A similar argument as before shows that $\pi^* T \mathbb{P}^1_k = \pi^* \Oh_{\mathbb{P}^1_k}(2)$ is relatively orientable over any field $k$. Consider the map $F: \A^2_k \to \A^2_k$ defined by $(x,y) \mapsto (y^2 - x^3 + x, x)$. We can calculate the Euler number $e(\Oh^2_{\A^2_k}, F)$ using the Jacobian of $F$, see \cite[Proposition 15]{kwEKL} and the discussion after \cite[Corollary 31]{kwcubic}. The Jacobian of $F$ is given by
\begin{equation*}
    \text{det} \begin{pmatrix} -3x^2 + 1 & 2y \\ 1 & 0 \end{pmatrix} = -2y.
\end{equation*}
The quadratic form $\beta: R \times R \to k$ associated to $e(\Oh^2_{\A^2_k}, F)$ is defined over the ring $R = k[x,y]/(y^2 - x^3 + x, x) \cong k[y]/(y^2)$. Recall that $\beta(a,b) = \eta(ab)$ for some k-linear morphism $\eta: R \to k$ such that $\eta(-2y) = \text{dim}_k R = 2$. Hence, the quadratic form $\beta$ is represented by
\begin{equation*}
    \begin{pmatrix} 0 & -1 \\ -1 & 0 \end{pmatrix},
\end{equation*}
which is equivalent to the class $\langle 1 \rangle + \langle -1 \rangle$ in the Grothendieck-Witt ring $GW(k)$, by Lemma \ref{bezoutlemma}. Recall that the hyperbolic element is invariant under multiplication by the class $\langle \alpha \rangle$. Using Corollary \ref{cor:curve_euler_number} and suitably choosing Nisnevich coordinates, we can hence show that there exists a unit $\alpha \in k^\times$ such that
\begin{align*}
    \EAdeg \pi &= e(\pi^* \Oh_{\Proj^1_k}(1), x) = e(\Oh_{\tilde{E}}, x) = \ind_{(0,0) \in \tilde{E}} x \\
    &= \langle \alpha \rangle \bullet \; \ind_{(0,0) \in \A^2_k} F = \langle \alpha \rangle \bullet e(\Oh_{\A^2_k}^2, F) \\
    &= \Adeg \pi.
\end{align*}
\end{example}

\section{Covering maps of modular curves} \label{sec: covering maps}

In this section, we compute both variants of the global $\A^1$ degrees of covering maps from coarse moduli schemes of elliptic curves with level structures to $X(1)$ under certain conditions. Using Proposition \ref{A1 degree double ramify}, we show that both the global Euler $\A^1$ degree and the global $\A^1$ degree of these covering maps are equivalent to hyperbolic elements in $\GW(k)$.

The below discussion on $\Gamma$-structures on elliptic curves and coarse moduli schemes of these $\Gamma$-structures is based on content in \cite{km85}.

By an elliptic curve over a scheme $S$, we mean a proper smooth curve over $S$ with geometrically connected fibers all of genus one and with a section $0$.

\begin{definition}
Let $E/S$ be an elliptic curve, and let $N \geq 1$ be an integer. 
\begin{enumerate}
    \item \cite[3.1]{km85} A $\Gamma(N)$-structure on $E / S$ is a group homomorphism $\phi:(\mathbb{Z} / N\mathbb{Z})^2 \rightarrow E[N](S)$ such that the Cartier divisor $\sum_{a, b \pmod N}[\phi(a, b)]$ equals $E[N]$. 

    \item \cite[3.2]{km85} A $\Gamma_1(N)$-structure on $E / S$ is a group homomorphism $\phi:\mathbb{Z} / N\mathbb{Z} \rightarrow E[N](S)$ such that the Cartier divisor $\sum_{a \pmod N}[\phi(a)]$ is a subgroup scheme of $E[N]$. 

    \item \cite[3.3]{km85} A $\Gamma_0(N)$-structure on $E/S$ is a finite flat subgroup scheme $C$ of $E[N]$ that is locally free of rank $N$ which is cyclic in the sense that $C$ admits a generator locally f.p.p.f. on $S$.
\end{enumerate}
We write $(E, (P,Q))$, $(E, P)$, and $(E, C)$ to respectively represent $\Gamma(N)$, $\Gamma_1(N)$, and $\Gamma_0(N)$-structures along with their underlying elliptic curves --- $(P,Q)$ in $(E, (P,Q))$ is the pair $(\phi(1,0), \phi(0,1))$, $P$ in $(E, P)$ is the image $\phi(1)$, and $C$ in $(E, C)$ is the subgroup scheme.
\end{definition}

Let $(\mathrm{Ell})$ denote the category whose objects are are elliptic curves $E/S$ over variable schemes $S$ and whose morphisms $E'/S' \to E/S$ are cartesian diagrams

\begin{equation}
\label{diagram:cartesian}
\begin{tikzcd}
E' \ar[r] \ar[d] & E \ar[d] \\
S' \ar[r] & S,
\end{tikzcd}
\end{equation}

cf. \cite[4.1]{km85}. For a ring $R$, let $(\mathrm{Ell}/R)$ denote the category of whose objects are elliptic curves over variable $R$-schemes and whose morphisms are cartesian diagrams \eqref{diagram:cartesian} whose bottom arrows are morphisms over $R$ \cite[4.13]{km85}. A moduli problem for $(\mathrm{Ell})$ is a contravariant functor $(\mathrm{Ell}) \to (\mathrm{Sets})$ \cite[4.2]{km85}. Similarly, a moduli problem for $(\mathrm{Ell}/R)$ is a contravariant functor $(\mathrm{Ell}/R) \to (\mathrm{Sets})$ \cite[4.13]{km85}. Given a moduli problem $\mathcal{P}$ for $(\mathrm{Ell})$, there is a base change moduli problem $\mathcal{P} \otimes R$ on $(\mathrm{Ell}/R)$ \cite[4.13]{km85}.

For each $\Gamma \in \{\Gamma(N), \Gamma_1(N), \Gamma_0(N)\}$, there is a moduli problem $[\Gamma]$ sending $E/S$ to the set of $\Gamma$-structures on $E/S$ \cite[5.1]{km85}. Each of these moduli problems is relatively representable over $(\mathrm{Ell})$, i.e. for every elliptic curve $E/S$ the functor on $(\mathrm{Sch}/S)$ defined by $T \mapsto [\Gamma](E_T/T)$ is representable by an $S$-scheme, and in fact its base change to $\mathbb{Z}[1/N]$ is a finite etale relatively representable moduli problem over $(\mathrm{Ell}/\mathbb{Z}[1/N])$, see \cite[Theorem 5.1.1]{km85}. In particular, $[\Gamma] \otimes \mathbb{Z}[1/N]$ is affine over $(\mathrm{Ell}/\mathbb{Z}[1/N])$ and hence has an associated coarse moduli scheme $M(\Gamma)$ \cite[8.1.1]{km85}. We denote these coarse moduli schemes of $[\Gamma(N)]$, $[\Gamma_1(N)]$, and $[\Gamma_0(N)]$ by $Y(N)$, $Y_1(N)$, and $Y_0(N)$ respectively. Writing $Y$ for the coarse moduli scheme of $\Gamma \in \{\Gamma(N), \Gamma_1(N), \Gamma_0(N)\}$, there is a $j$-morphism $Y \to Y(1) = \operatorname{Spec}(R[j])$ induced by the natural morphism $[\Gamma] \to [\Gamma(1)]$ of moduli problems forgetting the $\Gamma$-structure on the given elliptic curve $E/S$ \cite[8.2.1]{km85}.

In fact, $M(\Gamma)$ has a compactification $\overline{M}(\Gamma)$ constructed by \enquote{normalizing near infinity} \cite[8.6.3]{km85}. We denote $X(N)$, $X_1(N)$ and $X_0(N)$ such compactifications of $Y(N)$, $Y_1(N)$, and $Y_0(N)$ respectively. Writing $X$ for the compactification of $Y$, the $j$-morphism $Y \to Y(1)$ extends to a $j$-morphism $X \to X(1) = \mathbb{P}^1_j$.

Given a relatively representable moduli problem $\mathcal{P}$ on $(\mathrm{Ell}/R)$ and a group $G$ acting on $\mathcal{P}$, there is a notion of the quotient $\mathcal{P}/G$ of $\mathcal{P}$ by $G$ (\cite[7.1.2]{km85}). In fact, $\mathcal{P}$ always has a quotient by $G$ whenever $\mathcal{P}$ is affine over $(\mathrm{Ell}/R)$ and $G$ is finite (\cite[Theorem 7.3.1]{km85}). Note that the group action of $G$ induces a group action on $M(\mathcal{P})$ because $M(\mathcal{P})$ is constructed functorially. In fact, $M(\mathcal{P}) / G$ and $M(\mathcal{P} / G)$ are naturally isomorphic \cite[Lemma 8.1.5]{km85}.

For a given integer $N \geq 1$, the group $\operatorname{GL}(2, \mathbb{Z}/N\mathbb{Z})$ acts on $[\Gamma(N)]$ on the right as follows: an element $\begin{pmatrix} a & b \\ c & d \end{pmatrix} \in \operatorname{GL}(2,\mathbb{Z}/N\mathbb{Z})$ acts by sending a $\Gamma(N)$-structure $(P,Q)$ on an elliptic curve $E/S$ to $(P,Q) \begin{pmatrix} a & b \\ c & d \end{pmatrix} = (aP+cQ, bP+dQ)$. Moreover, $(\mathbb{Z}/N\mathbb{Z})^\times$ acts on $[\Gamma_1(N)]$; an element $a \in (\mathbb{Z}/N\mathbb{Z})^\times$ sends a $\Gamma_1(N)$-structure $P$ on an elliptic curve $E/S$ to $aP$. These actions describe how to regard some of these moduli problems as quotients of others.

\begin{theorem}[\cite{km85}, Theorem 7.4.2] \label{thm:moduli_problems_as_quotients_of_one_another}
Let $N \geq 1$ be an integer.

\begin{enumerate}

\item For any divisor $d$ of $N$, the natural map 
\begin{align*}
[\Gamma(N)] &\to [\Gamma(d)] \\
(E,(P,Q)) &\mapsto (E,((N/d)P, (N/d)Q))
\end{align*}
identifies $[\Gamma(d)]$ with the quotient of $[\Gamma(N)]$ by the principal congruence subgroup of $\operatorname{GL}_2 (\mathbb{Z}/N\mathbb{Z})$ consisting of the elements which are equivalent to the identity matrix modulo $d$.

In particular, $[\Gamma(1)]$ is the quotient of $[\Gamma(N)]$ by $\operatorname{GL}_2(\mathbb{Z}/N\mathbb{Z})$.

\item The natural map
\begin{align*}
[\Gamma(N)] &\to [\Gamma_1(N)] \\
(E,(P,Q)) &\mapsto (E,P)
\end{align*}
identifies $[\Gamma_1(N)]$ with the quotient of $[\Gamma(N)]$ by the \enquote{semi-Borel} subgroup of $\operatorname{GL}_2(\mathbb{Z}/N\mathbb{Z})$ consisting the elements of the form $\begin{pmatrix} 1 & * \\ 0 & * \end{pmatrix}$.

\item The natural map
\begin{align*}
[\Gamma_1(N)] &\to [\Gamma_0(N)] \\
(E, P) &\mapsto (E, \langle P \rangle)
\end{align*}
identifies $[\Gamma_0(N)]$ with the quotient of $[\Gamma_1(N)]$ by $(\mathbb{Z}/N\mathbb{Z})^\times$.

\end{enumerate}

\end{theorem}

To reiterate, there are covering maps $X(N), X_1(N), X_0(N) \to \mathbb{P}^1_j$ over $\mathbb{Z}[1/N]$. Given the above theorem and the above identification of $M(\mathcal{P})/G$ and $M(\mathcal{P}/G)$, there are also quotient maps $Y(N) \to Y_1(N)$ and $Y_1(N) \to Y_0(N)$ that in fact extend to quotient maps $X(N) \to X_1(N)$ and $X_1(N) \to X_0(N)$.

\begin{definition}\label{cover}
For any $N$, we denote by $\pi_{N}$, $\pi_{1,N}$, and $\pi_{0,N}$ the following covering maps of modular curves over $\Z[1/N]$:
\begin{align*}
    \pi_{N}&: X(N) \to X(1) \\
    \pi_{1,N}&: X_1(N) \to X(1) \\
    \pi_{0,N}&: X_0(N) \to X(1).
\end{align*}
\end{definition}

In fact, these maps factor as follows:

\begin{center}
\begin{tikzcd}
X(N) \ar[r, dashed] \ar[rrr, "\pi_N", bend left=75] &  X_1(N) \ar[r, dashed] \ar[rr, "\pi_{1,N}", bend left=50] &  X_0(N) \ar[r, "\pi_{0,N}"]  & X(1)
\end{tikzcd}
\end{center}

We also denote by $\pi_{N}$, $\pi_{1,N}$, and $\pi_{0,N}$ the respective morphisms $\pi_{N}: Y(N) \to Y(1)$, $\pi_{1,N}: Y_1(N) \to Y(1)$, and $\pi_{0,N}: Y_0(N) \to Y(1)$ as well as any of these morphisms based changed to any field of characteristic not dividing $N$.

We more precisely identify the Galois groups of the quotient maps of coarse moduli schemes suggested by Theorem \ref{thm:moduli_problems_as_quotients_of_one_another}.

\begin{lemma} \label{lem:kernel_of_actions_on_Y}
Let $N \geq 1$ be an integer. 
\begin{enumerate}
    \item The kernel of the action of $\operatorname{GL}_2(\mathbb{Z}/N\mathbb{Z})$ on $Y(N)$ is $\{ \pm I\}$. \label{lem:part:kernel_of_actions_on_Y_N}
    \item The kernel of the action of the semi-Borel subgroup of $\operatorname{GL}_2(\mathbb{Z}/N\mathbb{Z})$ on $Y(N)$ is $\{I\}$.
    \item The kernel of the action of $(\mathbb{Z}/N\mathbb{Z})^\times$ on $Y_1(N)$ is $\{\pm 1\}$.
\end{enumerate}
\begin{proof}
\begin{enumerate}
    \item  Given a relatively representable moduli problem $\mathcal{P}$ on $(\mathrm{Ell}/R)$ that is affine over $(\mathrm{Ell}/R)$ and given any algebraically closed field $\bar{K}$, the $\bar{K}$-points of $M(\mathcal{P})$ correspond to the $\bar{K}$-isomorphism classes of elliptic curves $E/\bar{K}$ with given $\mathcal{P}$-structure \cite[Lemma 8.1.3.1]{km85}. Therefore, the $\bar{K}$-points of $Y(N)$ correspond with the set of $\bar{K}$-isomorphism classes of elliptic curves $E/\bar{K}$ equipped with $\Gamma(N)$-structure. Note that the automorphism group of any elliptic curve over an algebraically closed field of $j$-invariant not $0$ or $1728$ is exactly $\{ \pm 1 \}$ \cite[Theorem III.10.1]{silverman}, so any element of $\operatorname{GL}_2(\mathbb{Z}/N\mathbb{Z})$ other than $\pm I$ acts nontrivially on any $\bar{K}$-point on $Y(N)$ whose underlying elliptic curve has $j$-invariant other than $0$ or $1728$. In particular, the kernel of the action is contained in $\{\pm I\}$. If $N \leq 2$, then $I = -I$ in $\operatorname{GL}_2(\mathbb{Z}/N\mathbb{Z})$, so the kernel is exactly $\{\pm I\}$. Now assume that $N \geq 3$.

    In contrast, the $\Gamma(N)$-structure $(E/\bar{K}, (P,Q))$ is isomorphic to $(E/\bar{K}, (-P, -Q)) = (E/\bar{K}, (P,Q)) \cdot (-I)$ because any elliptic curve $E/\bar{K}$ has an automorphism sending a point $P$ to $-P$. Therefore, the degree of any geometric fiber of $\pi_N$ is $\frac{1}{2} \#\operatorname{GL}_2(\mathbb{Z}/N\mathbb{Z})$. If the kernel of the action of $\operatorname{GL}_2(\mathbb{Z}/N\mathbb{Z})$ on $Y(N)$ were to be $\{I \}$, then the Galois group of $Y(N)$ over $Y(1)$ would be exactly $\operatorname{GL}_2(\mathbb{Z}/N\mathbb{Z})$, in which case the degree of any geometric fiber of $\pi_N$ would have to be $\# \operatorname{GL}_2(\mathbb{Z}/N\mathbb{Z})$, which is a contradiction. Hence, the kernel of the action must be $\{ \pm I\}$.
    
    \item This follows immediately from Part (\ref{lem:part:kernel_of_actions_on_Y_N}) because the semi-Borel subgroup of $\operatorname{GL}_2(\mathbb{Z}/N\mathbb{Z})$ contains $I$ but not $-I$ if $N \geq 3$ and because $I = -I$ in the semi-Borel subgroup if $N \leq 2$.

    \item We argue similarly as in Part (\ref{lem:part:kernel_of_actions_on_Y_N}). On the one hand, an element $a \in (\mathbb{Z}/N\mathbb{Z})^\times$ that is not $\pm 1$ sends the geometric point corresponding to the isomorphism class of $(E/\bar{K}, P)$ to the geometric point corresponding to the isomorphism class of $(E/\bar{K}, aP)$. The latter isomorphism class does not contain $(E/\bar{K}, P)$ whenever $E$ has $j$-invariant not $0$ or $1728$. Therefore, the kernel is contained in $\{\pm 1\}$. If $N \leq 2$, then $1 = -1$ in $(\mathbb{Z}/N\mathbb{Z})^\times$, so the kernel is exactly $\{\pm 1\}$; now assume that $N \geq 3$. 
    
    In contrast, $(E/\bar{K}, P)$ and $(E/\bar{K}, -P)$ are isomorphic, so the degree of any geometric fiber of $\pi_N$ must be $\# \frac{1}{2} (\mathbb{Z}/N\mathbb{Z})^\times$. If the kernel of the action of $(\mathbb{Z}/N\mathbb{Z})^\times$ on $Y_1(N)$ were to be $\{1\}$, then the Galois group of $Y_1(N)$ over $Y_0(N)$ would be exactly $(\mathbb{Z}/N\mathbb{Z})^\times$, in which case the degree of any geometric fiber of $\pi_1(N)$ would have to be $\# (\mathbb{Z}/N\mathbb{Z})^\times$, so the kernel of the action must be $\{ \pm 1\}$.
    
\end{enumerate}
\end{proof}
\end{lemma}

\begin{corollary} \label{cor:coarse_moduli_schemes_as_quotients_of_one_another}
Let $N \geq 1$ be an integer.

\begin{enumerate}
    \item $Y(1)$ is a quotient of $Y(N)$ by $\operatorname{GL}_2(\mathbb{Z}/N\mathbb{Z})$ and in fact the Galois group of the quotient map $Y(N) \to Y(1)$ is $\operatorname{GL}_2(\mathbb{Z}/N\mathbb{Z})/\{\pm I\}$.
    
    \item $Y_1(N)$ is a quotient of $Y(N)$ by the semi-Borel subgroup of $\operatorname{GL}_2(\mathbb{Z}/N\mathbb{Z})$ and in fact that Galois group of the quotient map $Y(N) \to Y_1(N)$ is the semi-Borel subgroup.
    
    \item $Y_0(N)$ is a quotient of $Y_1(N)$ by $(\mathbb{Z} / N\mathbb{Z})^\times$ and in fact the Galois group of the quotient map $Y_1(N) \to Y_0(N)$ is $(\mathbb{Z}/N\mathbb{Z})^\times / \{\pm 1\}$.
\end{enumerate}

\begin{proof}
These follow immediately from Lemma \ref{lem:kernel_of_actions_on_Y}, along with the identification of $M(\mathcal{P}) / G$ with $M(\mathcal{P}/G)$ whenever $\mathcal{P}$ is a relatively representable moduli problem over $(\mathrm{Ell}/R)$ as discussed before Theorem \ref{thm:moduli_problems_as_quotients_of_one_another}.
\end{proof}

\end{corollary}

\begin{proposition}\label{coverdeg}
Let $N$ be a positive integer. The degrees of $\pi_{N}$, $\pi_{1,N}$, and $\pi_{0,N}$ are:
\begin{align*}
    \deg \pi_N &= \begin{cases} N^4 \prod_{p | N, \; p \; \text{prime}} \left(1 - \frac{1}{p^2} \right) \left(1 - \frac{1}{p}\right) &\text{ if } N \neq 2 \\
    6 &\text{ if } N = 2 \end{cases} \\
    \deg \pi_{1,N} &= \frac{1}{2} N^2 \prod_{p | N, \; p \; \text{prime}} \left(1 - \frac{1}{p^2} \right)\\
    \deg \pi_{0,N} &= N \prod_{p | N, \; p \; \text{prime}} \left(1 + \frac{1}{p} \right)
\end{align*}
\end{proposition}
\begin{proof}
The first statement of the proposition follows from the fact that the kernel of the action of $\textrm{GL}_2(\mathbb{Z}/N\mathbb{Z})$ on $Y(N)$ is the element $\pm \begin{pmatrix} 1 & 0 \\ 0 & 1 \end{pmatrix}$ by Lemma \ref{lem:kernel_of_actions_on_Y}
\end{proof}

The ramification indicies of the fibers at $j = 1728$ play a crucial role in computing the global $\A^1$ degrees of the covering maps.

\begin{proposition}\label{ram}
Let $N$ be be a positive integer, and base change $\pi_N$, $\pi_{1,N}$, and $\pi_{0,N}$ to a field $K$ of characteristic not dividing $6N$.
\begin{enumerate}
    \item \label{ram:pi_N}
    For any $N \geq 2$, all the fibers of $\pi_N$ at $j = 1728$ are doubly ramified.
    
    \item \label{ram:pi_1_N}
     For any $N \geq 3$, all the fibers of $\pi_{1,N}$ at $j = 1728$ are doubly ramified.
     
    \item Let $N$ be a positive integer such that either $4 \mid N$ or there exists a prime $p \equiv 3 \; \text{mod} \; 4$ such that $p \mid N$. Then all the fibers of $\pi_{0,N}$ at $j = 1728$ are doubly ramified.
    
    \item Let $N \geq 3$ be a positive integer that does not satisfy both conditions in (3). Then there are $2^{\omega_\text{odd}(N)}$ fibers of $\pi_{0,N}$ at $j = 1728$ which are unramified where $\omega_{\text{odd}}(N)$ denotes the number of distinct odd prime factors of $N$. All other fibers are doubly ramified.
\end{enumerate}
\end{proposition}
\begin{proof}

Since ramifications of fibers are unaffected by geometric base change, we treat all fibral points geometrically in this proof.
\begin{enumerate}

\item Recall that the $\bar{K}$-points of $Y(N)$ correspond to the $\bar{K}$-isomorphism classes of elliptic curves $E/\bar{K}$ with given $\Gamma(N)$-structure \cite[Lemma 8.1.3.1]{km85}. In particular, the fibral points of $\pi_N$ over $j = 1728$ correspond to the pairs $(E_{j=1728}/\bar{K}, (P,Q))$ where $E_{j=1728}$ is the elliptic curve of $j$-invariant $1728$ and where $(P,Q)$ is a basis of $E_{j=1728}[N](\bar{K}) \cong (\mathbb{Z}/N\mathbb{Z})^2$. Since $\operatorname{char} K \nmid 6$, $E_{j=1728}$ has automorphism group of size $4$ over $\bar{K}$ \cite[Theorem III.10.1]{silverman}. In fact, $E_{j=1728}$ is given by the explicit equation $y^2 = x^3-x$, which has the automorphism $[i]: (x,y) \mapsto (-x, iy)$ of order $4$ where $i$ is some fixed square root of $-1$ in $\bar{K}$, so the automorphism group of $E_{j=1728}/\bar{K}$ is a cyclic group of order $4$.

Let $\gamma \in \operatorname{GL}_2(\mathbb{Z}/N\mathbb{Z})$ be the matrix such that $(E_{j=1728},(P,Q))\gamma = (E_{j=1728}, ([i]P, [i]Q))$. In particular, $\gamma$ is distinct from $\pm I$ and hence acts nontrivally on $Y(N)$ and yet $(E_{j=1728}, (P,Q))\gamma$ and $(E_{j=1728}, (P,Q)$ are isomorphic $\Gamma(N)$-structures, so they correspond to the same  geometric point of $Y(N)$. In fact, since the automorphism group of $E_{j=1728}$ is cyclic of order $4$ and generated by $[i]$, the elements of $\langle \gamma \rangle$ are the only elements of $\operatorname{GL}_2(\mathbb{Z}/N\mathbb{Z})$ to act trivially on $(E, (P,Q))$. Moreover, since $\pm \begin{pmatrix} 1 & 0 \\ 0 & 1 \end{pmatrix}$ act trivially on $Y(N)$ by Lemma \ref{lem:kernel_of_actions_on_Y}, the fibral points of $\pi_N$ over $j=1728$ are all doubly ramified.

\item Given Part (\ref{ram:pi_N}), it suffices to show that the covering map $X(N) \to X_1(N)$ has no ramifications above $j=1728$ in $X(1)$. By Corollary \ref{cor:coarse_moduli_schemes_as_quotients_of_one_another}, the Galois group of the covering map $Y_1(N) \to Y(N)$ is the semi-Borel subgroup of $\operatorname{GL}_2(\mathbb{Z}/N\mathbb{Z})$. The element $\begin{pmatrix} 1 & b \\ 0 & d \end{pmatrix}$ sends a geometric point corresponding to the $\Gamma(N)$-structure $(E_{j=1728}, (P,Q))$ to the geometric point corresponding to $(E_{j=1728}, (P, bP+dQ))$. The only $\Gamma(N)$-structures isomorphic to $(E_{j=1728}, (P,Q))$ are $(E_{j=1728}, (\pm P, \pm Q))$ and $(E_{j=1728}, (\pm [i] P, \pm [i] Q))$. Given that $N \geq 3$, the only one of these $\Gamma(N)$-structures that can coincide with $(E_{j=1728}, (P, bP+dQ))$ is $(E_{j=1728}, (P,Q))$ itself and only when $\begin{pmatrix} 1 & b \\ 0 & d \end{pmatrix} = I$. Therefore, only the identity element of the semi-Borel subgroup acts trivially on the fiber $Y(N)$ over $j=1728$, so $X(N) \to X_1(N)$ has no ramifications above $j=1728$ as desired.

\item Similarly, given Part (\ref{ram:pi_1_N}), it suffices to show that the covering map $X_1(N) \to X_0(N)$ has no ramifications above $j=1728$ in $X(1)$. 

Given that $4 \mid N$ or that $N$ has a prime factor that is congruent to $3$ modulo $4$, or equivalently, that $-1$ is not a square residue modulo $N$, the only way for $(E_{j=1728}/\bar{K}, P)$ to be isomorphic to $(E_{j=1728}/\bar{K}, aP)$ is for $a = \pm 1$. Indeed, the automorphism group of $E_{j=1728}/\bar{K}$ is generated by $[i]$ as established before, so $(E_{j=1728}/\bar{K}, P)$ and $(E_{j=1728}/\bar{K}, aP)$ can only be isomorphic via some power of $[i]$; if $(E_{j=1728}/\bar{K}, [i]^jP) = (E_{j=1728}/\bar{K}, aP)$ for some $a$ and some odd $j$, then $-P$ and $a^2P$ must coincide and hence $-1 \equiv a^2 \pmod{N}$, which cannot be the case.

Therefore, the Galois group $(\mathbb{Z}/N\mathbb{Z})^\times / \{\pm 1\}$ of $Y_1(N) / Y_0(N)$ acts faithfully on the fibers over $j=1728$ in $Y(1)$ and hence these fibral points are not ramified as desired.

\item Given Part (\ref{ram:pi_1_N}), it suffices to show that there are exactly $2^{\omega_{\text{odd}}(N)}$ geometric fibral points of $j = 1728$ via $\pi_{0,N}$ whose geometric fibral points via the quotient map $X_1(N) \to X_0(N)$ are doubly ramified. The geometric fibral points of $j=1728$ via $\pi_{0,N}$ are given by $(E_{j=1728}/\bar{K}, C)$ where $C$ is a cyclic subgroup of $E_{j=1728}[N](\bar{K})$ of exact order $N$. The automorphism $[i]$ on $E_{j=1728}$ induces an automorphism of $E[N](\bar{K}) \cong (\mathbb{Z}/N\mathbb{Z})^2$ with characteristic polynomial $T^2  + 1$. 

Let $\lambda$ be a square root of $-1$ in $(\mathbb{Z}/N\mathbb{Z})^\times$; note that there are $2^{\omega_{\text{odd}}(N)}$ such $\lambda$. By Lemma \ref{lem:i_is_diagonalizable_on_E_1728_N} below, the action of $[i]$ on $E_{j=1728}[N](\bar{K})$ is diagonalizable into $\begin{pmatrix} -\lambda & 0 \\ 0 & \lambda \end{pmatrix}$ if $N$ is odd. If $N$ is even, then $N$ is divisible by $2$ but not $4$ by assumption. Decompose $E_{j=1728}[N](\bar{K})$ into $E_{j=1728}[2](\bar{K}) \times E_{j=1728}[N/2](\bar{K})$ to recognize that the action of $[i]$ on $E_{j=1728}[2](\bar{K})$ is $\begin{pmatrix} 1 & 1 \\ 0 & 1 \end{pmatrix}$ with respect to some basis of $E_{j=1728}[2](\bar{K})$ by Lemma \ref{lem:i_is_1_1_0_1_on_E_1728_2} below and that the action of $[i]$ on $E_{j=1728}[N/2](\bar{K})$ is diagonalizable into $\begin{pmatrix} \lambda & 0 \\ 0 & -\lambda \end{pmatrix}$ by Lemma \ref{lem:i_is_1_1_0_1_on_E_1728_2}. In either case, $E[N](\bar{K})$ has cyclic submodule $C_\lambda$ of order $N$ on which $[i]$ acts as $\lambda$. Moreover, any element of $E[N](\bar{K})$ on which $[i]$ acts as $\lambda$ is an element of $C_\lambda$ --- if $N$ is odd, then $E[N](\bar{K})$ decomposes as $C_\lambda \oplus C_{-\lambda}$ and $[i]$ sends $P_\lambda + P_{-\lambda} \in C_\lambda \oplus C_{-\lambda}$ to $\lambda P_\lambda - \lambda P_{-\lambda}$, which equals $\lambda(P_\lambda + P_{-\lambda})$ exactly when $P_{-\lambda} = 0$, and if $N$ is even, then decompose $E_{j=1728}[N](\bar{K})$ into $E_{j=1728}[2](\bar{K}) \times E_{j=1728}[N/2](\bar{K})$ to recognize that $[i]$ acts as $\lambda$ on precisely a cyclic subgroup of order $2$ in $E_{j=1728}[2](\bar{K})$ and a cyclic subgroup of order $N/2$ in $E_{j=1728}[N/2](\bar{K})$ and hence $[i]$ acts as $\lambda$ on precisely a cyclic subgroup of order $N$ in $E_{j=1728}[2](\bar{K})$.

For each $\lambda$, let $P_\lambda$ be some generator of $C_\lambda$. The fiber of $(E_{j=1728}, C_\lambda)$ via $X_1(N) \to X_0(N)$ is the $(\mathbb{Z}/N\mathbb{Z})^\times$-orbit of $(E, P_\lambda)$. Note that $\lambda$ sends $(E_{j=1728}, P_\lambda)$ to $(E_{j=1728}, \lambda P_\lambda) = (E_{j=1728}, [i] P_\lambda)$, which is isomorphic to $(E_{j=1728}, P_\lambda)$. In particular, $\lambda$ acts trivially on $(E_{j=1728}, P_\lambda)$. In fact, only the powers of $\lambda$ (i.e. $\pm 1, \pm \lambda$) act trivially on $(E_{j=1728}, P_\lambda)$. Moreover, since $\pm 1$ act trivially on $X_1(N)$, $1$ and $\lambda$ are exactly the elements of $(\mathbb{Z}/N\mathbb{Z})^\times / \{\pm 1\}$ which act trivally on the fibral points of $(E_{j=1728}, C_\lambda)$ via $X_1(N) \to X_0(N)$, so these fibral points are doubly ramified. Note that there are $2^{\omega_{\text{odd}}(N)}$ of these $C_\lambda$. On the other hand, suppose that $C$ is a cyclic subgroup of $E[N](\bar{K})$ that is not any of the $C_\lambda$. The fibral points of $(E,C)$ via $X_1(N) \to X_0(N)$ are of the form $(E,P)$ where $P$ is a generator of $C$. Since $C$ is not any of the $C_\lambda$, no multiple of $P$ by an element of $(\mathbb{Z}/N\mathbb{Z})^\times$ equals $[i] P$ and hence $(\mathbb{Z}/N\mathbb{Z})^\times$ acts faithfully on the fiber of $(E,C)$. Thus, the fibral points of $(E,C)$ are unramified. Therefore, there are exactly $2^{\omega_{\text{odd}}(N)}$ geometric fibral points of $j=1728$ via $\pi_{0,N}$ whose geometric fibral points via $X_1(N) \to X_0(N)$ are doubly ramified as desired.

\end{enumerate}

\end{proof}

\begin{lemma} \label{lem:i_is_diagonalizable_on_E_1728_N}
Let $M$ be a $2 \times 2$ matrix with entries in $\mathbb{Z}/N\mathbb{Z}$ with $N$ an odd integer such that $\mathbb{Z}/N\mathbb{Z}$ has square roots of $-1$. If $M$ has characteristic polynomial $T^2+1$, then $M$ is diagonalizable over $\mathbb{Z}/N\mathbb{Z}$ into the matrix $\begin{pmatrix} \lambda & 0 \\ 0 & -\lambda \end{pmatrix}$, i.e. there is some $P \in \operatorname{GL}_2(\mathbb{Z}/N\mathbb{Z})$ such that $\begin{pmatrix} \lambda & 0 \\ 0 & -\lambda \end{pmatrix} = PMP^{-1}$.
\begin{proof}
By the Chinese remainder theorem, it suffices to show this in the case that $N$ is a prime power, say $N = p^\alpha$.

Write $M = \begin{pmatrix} a_{11} & a_{12} \\ a_{21} & a_{22} \end{pmatrix}$ where $a_{11} = -a_{22}$. First suppose that $a_{11} \not\equiv \lambda \pmod{p}$. In view of \cite[Theorem 2.3]{LAKSOV2013123}, it suffices to show that the determinant of the matrix denoted by $M(T)^c(\lambda, -\lambda)$ is a regular element of $\mathbb{Z}/N\mathbb{Z}$, i.e. is nonzero and is a nonzerodivisor. 

Following notations introduced in \cite{LAKSOV2013123}, $M(T)^c$ equals the matrix $\begin{pmatrix} a_{22} + \lambda & -a_{21} \\ -a_{12} & a_{11} - \lambda \end{pmatrix}$ which has determinant
$$(a_{22} + \lambda)(a_{11} - \lambda) - a_{12} \cdot a_{21} = a_{22} \cdot a_{11} + a_{11} \lambda - a_{22} \lambda + 1 - a_{12} \cdot a_{21} = 2 + 2 \lambda a_{11}.$$
Since $a_{11} \not\equiv \lambda \pmod{p}$, this determinant is invertible in $(\mathbb{Z}/N\mathbb{Z})^\times$.

Now suppose that $a_{11} \equiv \lambda(\pmod{p})$. Conjugating $M$ by $\begin{pmatrix} 0 & 1 \\ 1 & 0 \end{pmatrix}$ yields $\begin{pmatrix} a_{22} & a_{21} \\ a_{12} & a_{11} \end{pmatrix}$. Since $N$ is odd and since $a_{22} = -a_{11}$, $a_{22} \not\equiv \lambda \pmod{p}$. Thus, $M$ is actually diagonalizable into $\begin{pmatrix} \lambda & 0 \\ 0 & -\lambda \end{pmatrix}$.

\end{proof}
\end{lemma}

\begin{lemma} \label{lem:i_is_1_1_0_1_on_E_1728_2}
    Let $K$ be any field of characteristic not $2$ or $3$ and let $E_{j=1728}$ be an elliptic curve of $j$-invariant $1728$. There is a basis of $E_{j=1728}[2](\bar{K})$ with respect to which the automorphism $[i]$ on $E_{j=1728}$ acts by $\begin{pmatrix} 1 & 1 \\ 0 & 1 \end{pmatrix}$.
    \begin{proof}
        $E_{j=1728}$ can be given by the (affine) Weierstrass equation $y^2 = x^3 - x$, and $E_{j=1728}[2](\bar{K})$ has affine points $(\pm 1, 0)$ and $(0,0)$. Moreover, $[i]$ acts on $E_{j=1728}(\bar{K})$ by $(x,y) \mapsto (-x, iy)$ for some square root $i$ in $\bar{K}$. Thus, $[i]$ sends $(1, 0)$ to $(-1, 0)$, which equals $(1,0) + (0,0)$ and fixes $(0,0)$, so $[i]$ acts on $E_{j=1728}[2](\bar{K})$ as $\begin{pmatrix} 1 & 1 \\ 0 & 1 \end{pmatrix}$ with respect to the basis $(1,0), (0,0)$ of $E_{j=1728}[2](\bar{K})$.
    \end{proof}
\end{lemma}

Now we compute the global $\A^1$ degrees of covering maps of modular curves.

\begin{theorem}\label{mainA}
Given a fixed integer $N \geq 2$, let $k$ be any field such that $\text{char}(k) \nmid 6N$. Denote by $\pi \in \{\pi_{N}, \pi_{1,N}, \pi_{0,N} \}$ the covering maps of modular curves over $\mathbb{Z}[1/N]$ from Definition \ref{cover}.
\begin{enumerate}
    \item \label{mainA_Euler} Suppose the line bundle $\pi^*\Oh_{\Proj^1_k}(1)$ is relatively oriented. Then the global Euler $\A^1$ degrees of the covering maps are determined as follows.
    \begin{enumerate}[label=(\alph*)]
        \item \label{mainA_Euler_pi} For any $N \geq 2$,
        \begin{equation*}
            \EAdeg \pi_N = \begin{cases} \frac{1}{2} N^4 \prod_{p | N, \; p \; \text{prime}} \left(1 - \frac{1}{p^2} \right) \left(1 - \frac{1}{p} \right) \left( \langle 1 \rangle + \langle -1 \rangle \right) &\text{ if } N \neq 2 \\
        3(\langle 1 \rangle + \langle -1 \rangle) &\text{ if } N = 2. \end{cases}
        \end{equation*}
        \item \label{mainA_Euler_pi1} For any $N \geq 3$,
        \begin{equation*}
            \EAdeg \pi_{1,N} = \frac{1}{4} N^2 \prod_{p | N, \; p \; \text{prime}} \left(1 - \frac{1}{p^2} \right) \left(\langle 1 \rangle + \langle -1 \rangle \right).
        \end{equation*}
        \item \label{mainA_Euler_pi0} For any $N \geq 3$,
        \begin{equation*}
            \EAdeg \pi_{0,N} = \frac{1}{2} N \prod_{p | N, \; p \; \text{prime}} \left(1 + \frac{1}{p} \right) \left(\langle 1 \rangle + \langle -1 \rangle \right).
        \end{equation*}
    \end{enumerate}
    \item \label{mainA_A1} Suppose the line bundle $\pi^*T \mathbb{P}^1_k$ is relatively oriented. Then the global $\A^1$ degree of the covering maps are determined as follows.
    \begin{enumerate}[label=(\alph*)]
        \item \label{mainA_A1_pi} For any $N \geq 2$,
        \begin{equation*}
        \Adeg \pi_N = \begin{cases} \frac{1}{2} N^4 \prod_{p | N, \; p \; \text{prime}} \left(1 - \frac{1}{p^2} \right) \left(1 - \frac{1}{p} \right) \left( \langle 1 \rangle + \langle -1 \rangle \right) &\text{ if } N \neq 2 \\
        3(\langle 1 \rangle + \langle -1 \rangle) &\text{ if } N = 2. \end{cases}
        \end{equation*}
        \item \label{mainA_A1_pi1} For any $N \geq 3$,
        \begin{equation*}
        \Adeg \pi_{1,N} = \frac{1}{4} N^2 \prod_{p | N, \; p \; \text{prime}} \left(1 - \frac{1}{p^2} \right) \left(\langle 1 \rangle + \langle -1 \rangle \right).
        \end{equation*}
        \item \label{mainA_A1_pi0_hyper} Suppose $N \geq 3$. If either $4 | N$ or  there exists a prime factor $q$ of $N$ such that $q \equiv 3 \; \text{mod} \; 4$, then 
        \begin{equation*}
        \EAdeg \pi_{0,N} = \Adeg \pi_{0,N} = \frac{1}{2} N \prod_{p | N, \; p \; \text{prime}} \left(1 + \frac{1}{p} \right) \left(\langle 1 \rangle + \langle -1 \rangle \right).
        \end{equation*}
        \item \label{mainA_A1_pi0_nonhyper} Suppose $N \geq 3$. If $4$ does not divide $N$ and none of the prime factors $q$ of $N$ satisfy $q \equiv 3 \text{ mod } 4$, then 
        \begin{align*}
        \Adeg \pi_{0,N} &= \frac{1}{2}\left(N \prod_{p \mid N, \; p \; \text{prime}} \left( 1 + \frac{1}{p} \right) - 2^{\omega_{\text{odd}}(N)} \right) \left( \langle 1 \rangle + \langle -1 \rangle \right) \\
        & \; \; \; \; + \sum_{\substack{p \in \pi_{0,N}^{-1}(1728) \\ p \text{ has multiplicity } 1}} \Adeg_p \pi_{0,N},
        \end{align*}
        where $\omega_{\text{odd}}(N)$ is the number of distinct odd prime factors of $N$, and $\Adeg_p \pi_{0,N}$ is the local $\A^1$ degree of $\pi_{0,N}$ at the fibers $p \in \pi_{0,N}^{-1}(1728)$ which do not doubly ramify.
    \end{enumerate}
\end{enumerate}
\end{theorem}

\begin{proof}
We recall from Remark \ref{SW21} that the global Euler $\A^1$-degree of $\pi_{0,N}$ is equal to an integer multiple of the hyperbolic element. For the purpose of self containment, however, we compute the global $\A^1$-degrees explicitly. We first note that the conditions on $N$ from all the statements above are identical to those from Proposition \ref{ram}. Without loss of generality, we prove statements (\ref{mainA_Euler})-\ref{mainA_Euler_pi0}, (\ref{mainA_A1})-\ref{mainA_A1_pi0_hyper}, and (\ref{mainA_A1})-\ref{mainA_A1_pi0_nonhyper} of the theorem, whose proofs can be directly extended to other parts of statements (\ref{mainA_Euler}) and (\ref{mainA_A1}).

Let $N \geq 3$ be any positive integer. For each fiber $p$ of $\pi_{0,N}$ at $j = 1728$, let $k(p) / k$ be a finite field extension such that $p$, as a closed point, is defined over $k(p)$. For all such points $p$, note that $k(p)/k$ is a separable field extension because the covering map $\pi_{0,N}$ can be defined over the prime field of $k$ and such a prime field is always perfect as it is either $\mathbb{Q}$ or $\mathbb{F}_p$ where $p = \operatorname{char}(k)$. The global Euler $\A^1$ degree of $\pi_{0,N}$ is then given by
\begin{equation*}
    \EAdeg \pi_{0,N} = e(\pi_{0,N}, 1728) = \sum_{p \in \pi_{0,N}^{-1}(1728)} \ind_p (j-1728)
\end{equation*}
where $j - 1728$ is considered as a global section over $\pi^* \Oh_{\Proj^1_k}(1)$.

Suppose further that either $4$ divides $N$ or there exists a prime $q \equiv 3 \; \text{mod} \; 4$ which divides $N$. By Lemma \ref{lemma:fiber}, the zero locus of the section $j - 1728$ is the fiber of $\pi_{0,N}$ at $j = 1728$. Proposition \ref{euler double ramify} and Proposition \ref{coverdeg} then show that
\begin{equation*}
    \ind_p (j - 1728) = [k(p):k] (\langle 1 \rangle + \langle -1 \rangle).
\end{equation*}

Because every fiber of $\pi_{0,N}$ at $j=1728$ is doubly ramified, Proposition \ref{A1 degree double ramify} shows that

\begin{align*}
    \EAdeg \pi_{0,N} &= \frac{1}{2} \deg \pi_{0,N} \left( \langle 1 \rangle + \langle -1 \rangle \right) \\
    &= \frac{1}{2} N \prod_{p \mid N, \; p \; \text{prime}} \left( 1 + \frac{1}{p} \right) \left( \langle 1 \rangle + \langle -1 \rangle \right),
\end{align*}
which proves statement (\ref{mainA_Euler})-\ref{mainA_Euler_pi0}.

We note that the affine model of the coarse moduli space $X_0(N)$ can be obtained from the zero sets of the classical modular polynomials $\Phi_n(j(n\tau), j(\tau)) = 0$. By Theorem \ref{thm:two_notions_agree}, we obtain that the global $\A^1$ degree of $\pi_{0,N}$ agrees with the global Euler $\A^1$ degree of $\pi_{0,N}$, thus proving statement (\ref{mainA_A1})-\ref{mainA_A1_pi0_hyper}.

To prove statement (\ref{mainA_A1})-\ref{mainA_A1_pi0_nonhyper}, we use Proposition \ref{ram} that all but $2^{\omega_\text{odd}(N)}$ fibers of $\pi_{0,N}$ at $j = 1728$ doubly ramify. 

To compute the global $\A^1$ degree of $\pi_{0,N}$, we use the following two ideas from the proof of Theorem \ref{thm:two_notions_agree}: One, that the global $\A^1$ degree of $\pi$ is equal to the sum of local $\A^1$ degrees (Definition \ref{sum_local_A1}), which is equivalent to the EKL class of some morphism $F: \A^r \to \A^r$ (Theorem \ref{localEKL}): The other, that the hyperbolic element is invariant under multiplication by elements in $\GW(k)$ of form $\langle a \rangle$ for any $a \in k^\times$. This implies that the local $\A^1$ degrees of $\pi_{0,N}$ at doubly ramified fibers $p \in \pi_{0,N}^{-1}(1728)$ are equal to $[k(p):k](\langle 1 \rangle + \langle -1 \rangle)$. Proposition \ref{ram} hence implies that the global $\A^1$-degree of $\pi_{0,N}$ must contain the element $\frac{1}{2}\left(N \prod_{p \mid N, \; p \; \text{prime}} \left( 1 + \frac{1}{p} \right) - 2^{\omega_{\text{odd}}(N)} \right) \left( \langle 1 \rangle + \langle -1 \rangle \right)$ as a summand.
\end{proof}

In fact, when the modular curve is isomorphic to $\Proj^1_k$, we can compute the global $\A^1$ degrees of covering maps by computing B\'ezout matrices of explicitly constructed equations for covering maps, for instance as found from \cite{mcmurdy} for some modular curves $X_0(N)$, as per Theorem \ref{thm:globalA1Bezout} and by using Proposition \ref{curve_euler_number}. These techniques correspond to computing the local $\A^1$ degrees at fibers of covering maps at $j = 0$. Corollary \ref{euler deg independent} implies that the global $\A^1$ degrees computed at the fibers of $j = 0$ and $j = 1728$ are equal. 

The following list of examples shows how one can explicit compute the global $\A^1$ degrees of covering maps $\pi_{0,p}$ for some prime numbers $p$.
\begin{itemize}
    \item Example \ref{ex:pi_0,3}: We show that both the global Euler $\A^1$ degree and the global $\A^1$ degree of the covering map $\pi_{0,3}: X_0(3) \to X(1)$ are integers multiple of the hyperbolic element, both of which conform to Statements (\ref{mainA_Euler})-\ref{mainA_Euler_pi0} and (\ref{mainA_A1})-\ref{mainA_A1_pi0_hyper} of Theorem \ref{mainA}. 
    \item Example \ref{ex:pi_0,11}: We show that both the global Euler $\A^1$ degree and the global $\A^1$ degree of the covering map $\pi_{0,11}: X_0(11) \to X(1)$ are integers multiple of the hyperbolic element, both of which conform to Statements (\ref{mainA_Euler})-\ref{mainA_Euler_pi0} and (\ref{mainA_A1})-\ref{mainA_A1_pi0_hyper} of Theorem \ref{mainA}. 
    \item Remark \ref{nonex1}: We show that both the global Euler $\A^1$ degree and the global $\A^1$ degree of the covering map $\pi_{0,2}: X_0(2) \to X(1)$ are not well defined over any field $k$, which can be identified using the Atkin-Lehner involution.
    \item Remark \ref{nonex2}: We compute the global Euler $\A^1$ degree and the global $\A^1$ degree of the covering map $\pi_{0,5}: X_0(5) \to X(1)$, the results of which conform to Statement (\ref{mainA_A1})-\ref{mainA_A1_pi0_nonhyper} of Theorem \ref{mainA}.
\end{itemize}

\begin{example} \label{ex:pi_0,3}
Let $N = 3$. We first verify that both the global $\A^1$ degree and the global Euler $\A^1$ degree of $\pi_{0,3}$ are well-defined as long as the characteristic of the base field $k$ is coprime to $2$ and $3$. To see that the global Euler $\A^1$-degree is defined, it suffices to show that $\pi_{0,5}^* \Oh_{\mathbb{P}^1_k}(1)$ is a relatively orientable bundle of $\mathbb{P}^1_k$. Indeed, we have
$$\operatorname{Hom}(\det T \mathbb{P}^1_k, \det \pi_{0,3}^* \Oh_{\mathbb{P}^1_k}(1)) \cong \operatorname{Hom}(\Oh_{\mathbb{P}^1_k}(2), \Oh_{\mathbb{P}^1_k}(4)) \cong \Oh_{\mathbb{P}^1_k}(2) \cong \Oh_{\mathbb{P}^1_k}(1)^{\otimes 2}.$$ Theorem \ref{mainA} then implies that $\EAdeg \pi_{0,3} = 2 ( \langle 1 \rangle + \langle -1 \rangle)$. 

To see that the global $\A^1$-degree is well defined, we show that $\pi_{0,3}^* T \mathbb{P}^1_k$ is relatively orientable (The fact that $T\pi_{0,3}$ is invertible at some point, $\mathbb{P}^1_k$ is $\A^1$-chain connected is immediate). The relative orientability of the pullback of the line bundle follows from
$$\operatorname{Hom}(\det T \mathbb{P}^1_k, \det \pi_{0,3}^* \Oh_{\mathbb{P}^1_k}(2)) \cong \operatorname{Hom}(\Oh_{\mathbb{P}^1_k}(2), \Oh_{\mathbb{P}^1_k}(8)) \cong \Oh_{\mathbb{P}^1_k}(6) \cong \Oh_{\mathbb{P}^1_k}(3)^{\otimes 2}.$$

The covering map $\pi_{0,3}:X_0(3) \to X(1)$ is a morphism from $\Proj^1_k$ to itself, so we can compute the B\'ezout matrix associated to $\pi_{0,3}$ to compute its global $\A^1$ degree (which is the same as the global $\A^1$ degree by Theorem \ref{thm:globalA1Bezout}). The pullback $\pi_{0,3}^*(j)$ of the $j$-invariant can be written explicitly as
\begin{equation*}
    \pi_{0,3}^*(j) = \frac{(t+27)(t + 3)^3}{t}
\end{equation*}
where $t$ is the hauptmodul of $X_0(3)$.
 
Let $F(t) = (t+27)(t+3)^3$ and $G(t) = t$. The polynomial $F(x)G(y) - F(y)G(x)$ factorizes as
\begin{equation*}
    F(x) G(y) - F(y) G(x) = (x-y)(x^3y + x^2y^2 + xy^3 + 36x^2y + 36xy^2 + 270xy - 729).
\end{equation*}
Hence the B\'ezout matrix is
\begin{equation*}
    \begin{pmatrix} -729 & 0 & 0 & 0 \\ 0 & 270 & 36 & 1 \\ 0 & 36 & 1 & 0 \\ 0 & 1 & 0 & 0 \end{pmatrix}
\end{equation*}
which represents the element 
\begin{equation*}
    \Adeg \pi_{0,3} = \langle -729 \rangle + 2 \langle 1 \rangle + \langle -1 \rangle = 2(\langle 1 \rangle + \langle -1 \rangle)
\end{equation*}
of $\GW(k)$.

One may ask whether the Atkin-Lehner involution $w_3: X_0(3) \to X_0(3)$ affects the global $\A^1$ degree of $\pi_{0,3}$. Thankfully, such is not the case. We compute the B\'ezout matrix associated to the Atkin-Lehner involution $w_3$. The pullback of the hauptmodul $t$ under the involution $w_3$ can be written explicitly as
\begin{equation*}
    w_3^*(t) = \frac{729}{t}.
\end{equation*}
The B\'ezout matrix is given by the following $1 \times 1$ matrix
\begin{equation*}
    \begin{pmatrix} -729 \end{pmatrix}
\end{equation*}
which represents the element $\langle -1 \rangle$ of $\GW(k)$. Accordingly, we would expect $\pi_{0,3} \circ w_3$ to have $\A^1$-degree $\langle -1 \rangle \bullet 2( \langle 1 \rangle + \langle -1 \rangle)$ (cf. \cite[Lemma 4.8]{cazanave}), but multiplication by $\langle -1 \rangle$ does not change the global $\A^1$ degree of $\pi_{0,3}$, i.e. the global $\A^1$ degree of $\pi_{0,3}$ is invariant under the action of the Atkin-Lehner involution. 
\end{example}

\begin{example} \label{ex:pi_0,11}
Let $N = 11$. The covering map $\pi_{0, 11}: X_0(11) \to X(1)$ is a morphism from a genus $1$ curve $C$ to $\Proj_k^1$.

We first check what conditions are required to ensure that the global Euler $\A^1$ degree and the global $\A^1$ degrees are well defined. We have
\begin{equation*}
    \pi_{0,11}^* \Oh_{\Proj^1_k}(1) \cong \pi_{0,11}^* \Oh_{\Proj^1_k}((\infty)) = \Oh_{C}(11 \cdot (0) + (\infty)).
\end{equation*}
On the other hand, because $C$ is a genus $1$ curve, the tangent bundle $TC \cong \Oh_C$. This implies
\begin{equation*}
    \operatorname{Hom}(\det TC, \det \pi_{0,11}^* \Oh_{\mathbb{P}^1_k}(1)) \cong \operatorname{Hom}(\Oh_C, \Oh_{C}(11 \cdot (0) + (\infty))) \cong \Oh_{C}(11 \cdot (0) + (\infty)).
\end{equation*}
This implies that the pullback of the line bundle is relatively orientable exactly when the divisor class of $(0) - (\infty)$ is a square. Thus, the global Euler $\A^1$ degree is well defined if and only if $\Oh_C((0)-(\infty))$ is a square line bundle. This cannot happen over $k = \mathbb{Q}$, for example, because the Manin-Drinfeld theorem states that the torsion subgroup of the Jacobian variety $J_0(11)$ of modular curves is generated by $(0) - (\infty)$ (In fact, the similar condition holds for all odd primes $p$). Assuming that $\pi_{0,11}^* \Oh_{\mathbb{P}^1_k}(1)$ is relatively orientable over $k$, we can use Theorem \ref{mainA} to show that $\EAdeg \pi_{0,11} = 6(\langle 1 \rangle + \langle -1 \rangle)$.

Nevertheless, the global $\A^1$ degree of $\pi_{0,11}$ is always well defined as long as $k$ has discriminant coprime to $2$, $3$, and $11$. To see this, we note 
\begin{align*}
    \operatorname{Hom}(\det TC, \det \pi_{0,11}^* \Oh_{\mathbb{P}^1_k}(2)) &\cong \operatorname{Hom}(\Oh_C, \Oh_{C}(22 \cdot (0) + 2 \cdot (\infty))) \\ &\cong \Oh_{C}(11 \cdot (0) + (\infty))^{\otimes 2}.
\end{align*}

Pick $t = \left( \frac{\eta_1}{\eta_{11}} \right)^{12}$, and let $y$ be a variable which satisfies
\begin{equation*}
    y^2 - (t+6)(t^3-2t^2-76t-212) = 0.
\end{equation*}
The equation above defines an affine model of $X_0(11)$. Note that $k(j(\tau), j_{11}(\tau)) = k(t, y)$. The covering map $X_0(11) \to X(1)$ can be written explicitly as $(t,y) \mapsto c_0 - c_1 y$, where

\begin{align*}
    c_0 &= \frac{1}{2}t^{11} - \frac{187}{2}t^9 - 253 t^8 + 5720 t^7 + 28721 t^6 - 92092 t^5 \\
    &- 837892 t^4 - 933856 t^3 + 4126320 t^2 + 9924800 t + 4360000\\
    c_1 &= \frac{1}{2}(t-1)(t+5)(t+4)(t+2)(t-10)(t^2-2t-44)(t^2-20).
\end{align*}

Consider the morphism $\phi_{0, 11}: \A_k^2 \to \A_k^2$ such that 
\begin{equation*}
    \phi_{0,11}(t,y) = (y^2 - (t+6)(t^3-2t^2-76t-212), c_0 - c_1 y).
\end{equation*}
Let $(t^*, y^*) \in \pi_{0, 11}^{-1}(1728)$ be any choice of a fiber. Denote by $k(t^*, y^*)$ the separable field extension over $k$ generated by $t^*$ and $y^*$.

Proposition \ref{curve_euler_number} implies that for any fiber $(t^*, y^*) \in \pi_{0, 11}^{-1}(1728)$, there exists a unit $\alpha \in k(t^*, y^*)$ such that
\begin{equation*}
    \ind_{(t^*, y^*)} (j-1728) = \langle \alpha \rangle \bullet \ind_{(t^*, y^*)} (j-1728).
\end{equation*}
We note that if $y = \frac{c_0 - 1728}{c_1}$, then the polynomial $y^2 - (t+6)(t^3 - 2t^2 - 76t - 212)$ can be reexpressed as
\begin{equation*}
    \frac{4(7641728 + 6475200t + 2113680t^2 + 325856 t^3 + 22692 t^4 + 492t^5 - t^6)^2}{c_1^2}.
\end{equation*}
Set $\alpha_{11} = 7641728 + 6475200t + 2113680t^2 + 325856 t^3 + 22692 t^4 + 492t^5 - t^6$. 

Calculating the local index of $\phi_{(0,11)}$ at $(t^*, y^*) \in f^{-1}(0, 1728)$ is equivalent to calculating the EKL class of $\phi_{0, 11}$ over the localization of the ring $R$ at $(t^*, y^*)$. The ring $R$ is given as
\begin{align*}
    R &= k(t^*, y^*) \left[ t, \frac{c_0 - 1728}{c_1} \right] / (y^2 - (t+6)(t^3 - 2t^2 - 76t - 212)) \\
    &= k(t^*, y^*) \left[ t, \frac{1}{c_1} \right] / \left( \frac{4 \alpha_{11}^2}{c_1^2} \right) \\
    &= k(t^*, y^*) \left[ t, \frac{1}{c_1} \right] / \left( \frac{4 \alpha_{11}^2}{c_1} \right).
\end{align*}

Since every root of $\alpha_{11} = 0$ doubly ramifies, we have
\begin{equation*}
    \langle \alpha \rangle \bullet \ind_{(t^*, y^*)} (j-1728) = \langle 1 \rangle + \langle -1 \rangle
\end{equation*}
for any $\alpha \in k(t^*, y^*)^\times$. Using the definition of the global $\A^1$ degree, we conclude that
\begin{equation*}
    \Adeg \pi_{0, 11} = \sum_{(t^*, y^*) \in \pi_{0, 11}^{-1}(1728)} [k(t^*,y^*) : k](\langle 1 \rangle + \langle -1 \rangle) = 6(\langle 1 \rangle + \langle -1 \rangle).
\end{equation*}
\end{example}

\begin{remark}\label{nonex1}
When $N = 2$, we demonstrate that both global Euler $\A^1$ degrees is not well defined by virtue of showing that $\pi_{0,2}^* \Oh_{\mathbb{P}^1_k}(1)$ is not relatively orientable. Indeed, we have
$$\operatorname{Hom}(\det T \mathbb{P}^1_k, \det \pi_{0,2}^* \Oh_{\mathbb{P}^1_k}(1)) \cong \operatorname{Hom}(\Oh_{\mathbb{P}^1_k}(2), \Oh_{\mathbb{P}^1_k}(3)) \cong \Oh_{\mathbb{P}^1_k}(1),$$
but the line bundle $\Oh_{\mathbb{P}^1_k}(1)$ is not a square for any field $k$.

The conditions necessary for the existence of global $\A^1$ degrees are all satisfied. It suffices to show that the pullback $\pi_{0,2}^* T \mathbb{P}^1_k$ is relatively orientable. This is because
$$\operatorname{Hom}(\det T \mathbb{P}^1_k, \det \pi_{0,2}^* T \mathbb{P}^1_k) \cong \operatorname{Hom}(\Oh_{\mathbb{P}^1_k}(2), \Oh_{\mathbb{P}^1_k}(4)) \cong \Oh_{\mathbb{P}^1_k}(2) \cong \Oh_{\mathbb{P}^1_k}(1)^{\otimes 2}.$$

One may attempt to compute the B\'ezout matrix associated to the covering map $\pi_{0,2}:X_0(2) \to X(1)$, but the Atkin-Lehner involution changes the $\A^1$ degree of $\pi_{0,2}$.

Let $k$ be any field whose characteristic is coprime to $2$. Let $t$ be a hauptmodul of $X_0(2)$. Then the pullback of the $j$-invariant $\pi_{0,2}^{*}(j)$ can be written explicitly as
\begin{equation*}
    \pi_{0,2}^*(j) = \frac{(t + 16)^3}{t}.
\end{equation*}
The B\'ezout matrix associated to $\pi_{0,2}$ is the matrix
\begin{equation*}
    \begin{pmatrix} -4096 & 0 & 0 \\ 0 & 1 & 1 \\ 0 & 1 & 0 \end{pmatrix}
\end{equation*}
whose corresponding element in $\GW(k)$ is $\langle -4096 \rangle + \langle 1 \rangle + \langle -1 \rangle = 2 \langle 1 \rangle + \langle -1 \rangle$. However, after applying the Atkin-Lehner involution, the pullback of the $j$-invariant can also be written as
\begin{equation*}
    \pi_{0,2}^*(j) = \frac{(t+256)^3}{t^2}.
\end{equation*}
The B\'ezout matrix associated to $\pi_{0,2}$ is
\begin{equation*}
    \begin{pmatrix} 0 & -16777216 & 0 \\ -16777216 & -196608 & 0 \\ 0 & 0 & 1 \end{pmatrix}
\end{equation*}
inducing the element $2\langle -1 \rangle + \langle 1 \rangle$ of $\GW(k)$. 

We now verify that the Atkin-Lehner involution $w_2: X_0(2) \to X_0(2)$ affects the global $\A^1$ degree of $\pi_{0,2}$. The pullback of the hauptmodul $t$ under the involution $w_2$ can be written as
\begin{equation*}
    w_2^*(t) = \frac{4096}{t}
\end{equation*}
which corresponds to the class $\langle -1 \rangle$. We note that the two $\A^1$ degrees of $\pi_{0,2}$ are related via multiplication by $\langle -1 \rangle$:
\begin{equation*}
    \langle -1 \rangle (2 \langle 1 \rangle + \langle -1 \rangle) = 2 \langle -1 \rangle + \langle 1 \rangle.
\end{equation*}

We hence observe that the global $\A^1$ degree of $\pi_{0,2}$ depends on the choice of the pullback of the $j$-invariant, thus is not well defined.
\end{remark}

\begin{remark}\label{nonex2}
Suppose that $N$ is a positive integer that is not divisible by $4$ and that has no prime factors which are $3$ modulo $4$.  Then the global $\A^1$ degree of $\pi_{0,N}: X_0(N) \to \mathbb{P}^1$ may not be equal to the global Euler $\A^1$ degree of $\pi_{0,N}$, because the latter is equal to an integer multiple of the hyperbolic element \cite[Proposition 19]{SW21} whereas the same might not hold for the former.

For instance, suppose $N = 5$. Let $t$ be a hauptmodul of $X_0(5)$, which is of genus $0$. The pullback of the $j$-invariant is given by
\begin{equation*}
    \pi_{0,5}^*(j) = \frac{(t^2+10t+5)^3}{t}.
\end{equation*}
In particular, $\pi_{0,5}$ is a degree $6$ map $\mathbb{P}^1 \to \mathbb{P}^1$.

We establish that both the global $\A^1$ degree and the global Euler $\A^1$ degree of $\pi_{0,5}$ are defined regardless of which base field $k$ of characteristic not $2$ or $5$ we work over. To see that the global Euler $\A^1$-degree is defined, it suffices to show that $\pi_{0,5}^* \Oh_{\mathbb{P}^1_k}(1)$ is a relatively orientable bundle of $\mathbb{P}^1_k$ --- indeed,
$$\operatorname{Hom}(\det T \mathbb{P}^1_k, \det \pi_{0,5}^* \Oh_{\mathbb{P}^1_k}(1)) \cong \operatorname{Hom}(\Oh_{\mathbb{P}^1_k}(2), \Oh_{\mathbb{P}^1_k}(6)) \cong \Oh_{\mathbb{P}^1_k}(4) \cong \Oh_{\mathbb{P}^1_k}(2)^{\otimes 2}.$$ In particular, further assuming that $k$ is not of characteristic $3$, Theorem \ref{mainA} applies to demonstrate that $\EAdeg \pi_{0,5} = 3 ( \langle 1 \rangle + \langle -1 \rangle)$. 

To see that the global $\A^1$-degree is defined, it suffices to show that $T\pi_{0,5}$ is invertible at some point, that $\mathbb{P}^1_k$ is $\A^1$-chain connected, and that $\pi_{0,5}$ is relatively orientable as a morphism $\mathbb{P}^1_k \to \mathbb{P}^1_k$, i.e. $\pi_{0,5}^* T\mathbb{P}^1_k \cong \pi_{0,5}^* \Oh_{\mathbb{P}^1_k}(2)$ is relatively orientable as a bundle on $\mathbb{P}^1_k$. The former two conditions are immediate, whereas the latter condition holds because
$$\operatorname{Hom}(\det T \mathbb{P}^1_k, \det \pi_{0,5}^* \Oh_{\mathbb{P}^1_k}(2)) \cong \operatorname{Hom}(\Oh_{\mathbb{P}^1_k}(2), \Oh_{\mathbb{P}^1_k}(12)) \cong \Oh_{\mathbb{P}^1_k}(10) \cong \Oh_{\mathbb{P}^1_k}(5)^{\otimes 2}.$$

To compute $\Adeg \pi_{0,5}$, we compute the B\'ezout matrix associated to $\pi_{0,5}$, which is the matrix
\begin{equation*}
    \begin{pmatrix} -125 & 0 & 0 & 0 & 0 & 0 \\ 0 & 1575 & 1300 & 315 & 30 & 1 \\ 0 & 1300 & 315 & 30 & 1 & 0 \\ 0 & 315 & 30 & 1 & 0 & 0 \\ 0 & 30 & 1 & 0 & 0 & 0 \\ 0 & 1 & 0 & 0 & 0 & 0 \end{pmatrix}
\end{equation*}
which induces the element in $\GW(k)$ is $3\langle 1 \rangle + 2\langle -1 \rangle + \langle -5 \rangle$. This element equals $3 \langle 1 \rangle + 3 \langle -1 \rangle$ exactly when $\langle -1 \rangle = \langle -5 \rangle$, which happens exactly when $5$ is a square in $k^\times$. 

We also note that the Atkin-Lehner involution $w_5$ of $X_0(5)$ pulls back the $j$-invariant as
\begin{equation*}
    w_5^*(j) = \frac{125}{t}.
\end{equation*}
The $\A^1$ degree of the Atkin-Lehner involution is $\langle -5 \rangle$. Note that the Atkin-Lehner involution still preserves the global $\A^1$ degree of $\pi_{0,5}$.
\end{remark}

\section{Future directions}
We end the paper with a discussion of two questions that naturally arise from previous sections.
\begin{enumerate}
    \item Let $\pi: X(\Gamma) \to X(1)$ be the covering map of modular curves where $\pi^* \Oh(1)$ is relatively oriented over the field $k$. Let $Y(\Gamma) := \pi^{-1}(\A^1_k)$.
    
    Using Proposition \ref{curve_euler_number_2}, Theorem \ref{mainA} implies that the global $\A^1$ degree of $\pi: X(\Gamma) \to X(1)$ is equal to the Euler number of $\Oh_{Y(\Gamma)}$ with respect to the $j$-invariant, considered as a weakly holomorphic modular form of weight $0$. It is hence a natural question to ask which arithmetic properties do Euler numbers of line bundles with respect to holomorphic modular forms of weight $2k$ over $X(\Gamma)$ or those with respect to weakly holomorphic modular forms of weight $2k$ over $Y(\Gamma)$ imply.
    
    \item In this paper, we viewed modular curves as coarse moduli schemes over $\Z[1/N]$. However, we may also consider moduli stacks of elliptic curves with level structures. Kobin and Taylor \cite{ktstack} extended $\A^1$ homotopy theory and the construction of Euler numbers to root stacks. In a similar manner, we may extend the definition of $\A^1$ degrees for morphisms of stacks. Then, we can compute the $\A^1$ degrees of covering maps $\pi_\Gamma: \mathcal{M}(\Gamma) \to \mathcal{M}_{1,1}$ considered as morphisms of moduli stacks of elliptic curves with level structures, where $\Gamma$ is any congruence subgroup of $\textrm{SL}_2(\Z)$. Lastly, we may ask how global $\A^1$ degrees of covering maps between coarse moduli schemes and $\A^1$ degrees of covering maps between moduli stacks compare to each other.
\end{enumerate} 

\subsection*{Acknowledgements}

We sincerely thank Jordan Ellenberg for suggesting the problem and patiently giving constructive comments and encouragement to us. We would like to thank Kirsten Wickelgren for giving a very enlightening talk at the 2019 Arizona Winter School on $\A^1$ enumerative geometry, as well as for giving invaluable feedback via email. We would like to thank Libby Taylor and Andrew Kobin for sharing their progress on extending $\A^1$ homotopy theory to root stacks and for giving helpful feedback and discussions via email. The second author would like to thank Junhwa Jung for helpful discussions on calculating traces of Grothendieck-Witt group elements.

Hyun Jong Kim was partially supported by the National Science Foundation Award DMS-1502553. Sun Woo Park was partially supported by National Institute for Mathematical Sciences (NIMS) grant funded by the Korean Government (MSIT) B20810000.

\nocite{*}
\bibliographystyle{alpha}
\bibliography{main}

\end{document}